\documentclass[11pt]{article}
\usepackage[english]{babel}
\usepackage{amsmath,amssymb}
\usepackage{mathtools}
\usepackage{bbm}
\usepackage{graphicx}
\usepackage{todonotes}
\usepackage{color}
\usepackage{multirow}
\usepackage{tikz}
\numberwithin{equation}{section}
\newcounter{todocounter}

\newtheorem{theorem}{Theorem}

\newtheorem{lemma}[theorem]{Lemma}

\newtheorem{remark}[theorem]{Remark}

\usepackage[ulem=normalem]{changes}
\definecolor{darkgreen}{rgb}{0.0, 0.4, 0.0}
\definechangesauthor[name={Andrea Thomann},color=red]{AT}
\setauthormarkup{}    
	
\title{
	\bf{An all speed second order IMEX relaxation scheme for the Euler equations}
}
\author{
	Andrea Thomann\footnote{Dipartimento di Scienze e Alta Tecnologia, Universit\`a degli Studi dell'Insubria, Via Valleggio 11, 22100 Como, Italy} \footnote{Correspondence to: Andrea Thomann, University of Insubria, Via Valleggio 11, 22100 Como (CO), Italy, Email: acthomann@uninsubria.it} \footnote{Marie Sk\l odowska-Curie fellow of the Istituto Nazionale di Alta Matematica Francesco Severi, Rome, Italy},
Markus Zenk\footnote{Fakult\"at f\"ur Mathematik und Informatik, Universit\"at W\"urzburg, Emil-Fischer-Str. 40, 97074 W\"urzburg, Germany},
Gabriella Puppo\footnote{Dipartimento di Matematica, La Sapienza Universit\`a di Roma, Piazzale Aldo Moro 5, 00185 Roma, Italy},\\ Christian Klingenberg\footnotemark[4]
}
\date{\today}
	
\begin{document}
\maketitle

\section*{Abstract}
We present an implicit-explicit finite volume scheme for the Euler 
equations. 
We start from the non-dimensionalised Euler equations where we split the pressure in a slow and a fast acoustic part. 
We use a Suliciu type relaxation model which we split in 
an explicit part, solved using a Godunov-type scheme based on an approximate 
Riemann solver, and an implicit part where we solve an elliptic equation 
for the fast pressure. 
The relaxation source terms are treated projecting the solution on the equilibrium manifold. 
The proposed scheme is positivity preserving with respect to the density and internal energy and asymptotic preserving towards the incompressible Euler equations. 
For this first order scheme we give a second order extension which maintains the positivity property. 
We perform numerical experiments in 1D and 2D to show the applicability of the proposed splitting and give convergence results for the second order extension.

\paragraph{Keywords} finite volume methods, Euler equations, positivity preserving, asymptotic preserving, relaxation, low Mach scheme, IMEX schemes

\section{Introduction}
We consider the non-dimensional Euler equations in $d$-space dimensions which are given by the following set of equations \cite{Dellacherie2010}
\begin{align}
\begin{split}
\rho_t + \nabla \cdot (\rho \mathbf{u}) &= 0 \\
(\rho u)_t + \nabla \cdot (\rho \mathbf{u} \otimes \mathbf{u} + \frac{1}{M^2} p \mathbb{I}) & = 0 \\
E_t + \nabla \cdot(\mathbf{u}(E + p)) &= 0, \\
\end{split}
\label{sys:nondimEuler}
\end{align}
where the total energy is given by 
\begin{equation}
E = \rho e + \frac{1}{2} M^2 \rho |u|^2
\label{eq:Etot}
\end{equation}
and $e > 0$ denotes the internal energy. 
The density is denoted by $\rho > 0$, $\mathbf{u} \in \mathbb{R}^d$ is the velocity vector and $M$ is a given Mach number which controls the ratio between the velocity of the gas and the sound speed.
Depending on the magnitude of the Mach number the characteristic nature of the flow changes. 
This makes the numerical simulation of these flows very challenging but also a very interesting research subject with a wide range of applications, for example in astrophysical stellar evolution or multiphase flows \cite{MiczekRoepkeEdelmann2015,AbbateIolloPuppo2019}.
For large Mach numbers, the flow is governed by compressible effects whereas in the low Mach limit the compressible equations converge towards the incompressible regime.
This behaviour was studied for example in \cite{klainermanMajda1982,Dellacherie2010,Schochet2005}.
We refer to \cite{FeireislKlingenbergMarkfelder2019} for a study of the full Euler equations.

Standard schemes designed for compressible flows like the Roe scheme \cite{Roe1981} or Godunov type schemes fail due to exessive diffusion when applied in the low Mach regime. 
A lot of work is dedicated to cure this defect, see for instance \cite{GuillardViozat1999,Turkel1987,Dellacherie2010,MiczekRoepkeEdelmann2015}.
Another way to ensure accurate solutions in the low Mach regime is the development of asymptotic preserving schemes which are consistent with its limit behaviour as $M$ tends to zero, see for example \cite{CordierDegondKumbaro2012,DimarcoLoubereVignal2017,BoscarinoRussoScandurra2018} and references therein.

Due to the hyperbolic nature of \eqref{sys:nondimEuler} the time step for an explicit scheme is restricted by a stability CFL condition that depends on the inverse of the fastest wave speed. 
In the case of \eqref{sys:nondimEuler} the acoustic wave speeds tend to infinity as $M$ tends to zero which leads to very small time steps to guarantee the stability of the explicit scheme. 
As a side effect all waves will be resolved by the explicit scheme although the fast waves are not necessary to capture the motion of the fluid as they carry a negligible amount of energy. 
Implicit methods on the other hand allow larger time steps but introduce diffusion on the slow wave which leads to a loss of accuracy. 
In addition at each iteration a non-linear often ill conditioned algebraic system has to be solved.  
Implicit-explicit (IMEX) methods try to overcome those disadvantages by treating the stiff parts implicitly and thus allow for a Mach number independent time step.
Many of those schemes are based on a splitting of the pressure in the spirit of Klein \cite{Klein1995} since the stiffness of the system is closely related with the pressure, see for example \cite{CordierDegondKumbaro2012,NoelleEtAl2014,Klein1995}.

Another way to avoid solving non-linear implicit systems is using a relaxation approach. 
See \cite{JinXin1995,suliciu1990,ChenLevermoreLiu1994,Liu1987,Natalini1996} for references on relaxation.
The idea of using relaxation is to linearise the equations which makes it easier to solve them implicitly as done in \cite{AbbateIolloPuppo2017} through Jin Xin relaxation \cite{JinXin1995}.
The linear degenerate structure of the resulting relaxation models allows to use accurate Godunov-type finite volume methods due to the knowledge about the Riemann solution also with implicit schemes as done in \cite{BerthonKlingenbergZenk2018}.
Here, we want to combine a Suliciu type relaxation with an IMEX scheme and the splitting of the pressure. 
The relaxation model we use is based on relaxing the slow and the fast pressure separately. 
The flux is then split such that only some relaxation variables are treated implicitly which leads to solving only one equation implicitly.
This results in a computationally cheap scheme.
We do not split inside the flux of the Euler equations but treat all terms explicitly.
This leads to a conservative explicit part for which we use a Godunov-type approximate Riemann solver.
The use of a Riemann solver allows us to prove with little effort the positivity preserving property of the scheme which is important in physical applications. 
In addition the way the flux is splitted leads to an asymptotic preserving scheme.

To have a relevant scheme for applications an extension to second or higher order is necessary. 
Higher order schemes in time can be achieved by using IMEX Runge Kutta (RK) methods as in \cite{BoscarinoRussoScandurra2018,DimarcoEtAl2018}. 
As standard in finite volume schemes higher order in space is achieved by reconstructing the cell interface values using WENO schemes for example \cite{Shu1998,CraveroPuppoSempliceVisconti2018}.
We use a MUSCL approach \cite{Berthon2005} to achieve a second order extension to the first order scheme which preserves the positivity of density and internal energy. 

The paper is organized as follows. 
In the next section we describe the relaxation model that is used to derive the IMEX scheme. 
Section \ref{sec:time-semi} is dedicated to the splitting of the flux into implicit and explicit terms and the structure of the first order time semi-discrete scheme.
The asymptotic preserving property is proved in Section \ref{sec:AP}. 
Next, we give the derivation of the fully discrete scheme in Section \ref{sec:fully_discrete}.
Therein the Godunov type scheme for the explicit part is given as well as the proof of the positivity of the resulting numerical scheme. 
In addition, we show that the diffusion introduced by the Riemann solver is independent of the Mach number due to the splitting described in Section \ref{sec:time-semi}.
In Section \ref{sec:Second} we give a second order extension for the first order scheme for which we show that it preserves the positivity property. 
It is followed by a section of numerical results to validate the theoretical results. 
A section of conclusion completes this paper.

\section{Suliciu Relaxation model}
To simplify the non-linear structure of the Euler equations \eqref{sys:nondimEuler} we make use of the Suliciu relaxation approach \cite{suliciu1990,Bouchut2004,CoquelPerthame1998} and references therein.
Compared to the Jin Xin relaxation \cite{JinXin1995} the original system of equations remains unchanged. 
Whereas in the Jin Xin relaxation approach the linearisation of the equations is achieved by relaxing every component of the flux function, in the Suliciu type relaxation single variables are relaxed in a fashion that is tailored to the problem.
This leads to a reduced diffusion introduced by the relaxation process compared to the Jin Xin relaxation.  

Following the usual Suliciu relaxation procedure, the key element is the relaxation of the pressure by introducing a new variable $\pi$,   
\begin{equation}
	(\rho \pi)_t + \nabla \cdot (\rho \pi \mathbf{u}) + a^2 \nabla \cdot \mathbf{u} = \frac{\rho}{\epsilon} (p - \pi).
	\label{eq:SuliciuPi}
\end{equation}
Here $\varepsilon$ denotes the relaxation time. 
Equation \eqref{eq:SuliciuPi} is derived by multiplying the density equation by $\partial_\rho p$. The non-linearity $\rho^2 \partial_\rho p(\rho,e)$ that thereby arises is replaced by a constant parameter $a^2$, in the following called the relaxation parameter.
To guarantee a stable diffusive approximation of the original Euler equations the relaxation parameter must meet the sub-characteristic condition $a > \rho \sqrt{\partial_\rho p(\rho,e)} > 0$ \cite{ChenLevermoreLiu1994}.
We want to profit from the properties of the Suliciu relaxation also for the non-dimensional Euler equations. 
In the following, we describe the relaxation model that was first introduced in \cite{BerthonKlingenbergZenk2018}.
Following \cite{KleinEtAl2001}, in a first step the pressure is decomposed into a slow dynamics and a fast acoustic component
\begin{equation*}
\frac{p}{M^2} = p + \frac{1-M^2}{M^2}p.
\end{equation*}
Approximating the slow and the fast pressure in the momentum equation in \eqref{sys:nondimEuler} by the variables $\pi$ and $\psi$ respectively,  momentum equation becomes
\begin{equation*}
	(\rho \mathbf{u})_t + \nabla \cdot \left(\rho \mathbf{u}\otimes\mathbf{u} + \pi + \frac{1-M^2}{M^2} \psi\right) = 0
\end{equation*}
The evolution of the new variables $\pi$ and $\psi$ are then developed in the spirit of a Suliciu relaxation approach.
The evolution for $\pi$ is given by the Suliciu relaxation equation \eqref{eq:SuliciuPi}.
However, applying the standard Suliciu relaxation method also on the pressure $\psi$, would only lead to non relevant diffusion terms and not to a low Mach scheme.
To overcome this, the authors of \cite{BerthonKlingenbergZenk2018} introduced a new velocity variable $\mathbf{\hat{u}} \in \mathbb{R}^{d}$ which is relaxed to the system velocity $\mathbf{u}$ and couples to the pressure $\psi$.
The form of the evolution equations for $\mathbf{\hat{u}}$ and $\psi$ is chosen, such that 
\begin{itemize}
	\item the resulting model is in conservation form
	\item the resulting model has ordered eigenvalues, which results in a clear wave structure
	\item the resulting model is a stable diffusive approximation of the non-dimensional Euler equations \eqref{sys:nondimEuler}
	\item the resulting numerical scheme has a Mach number independent diffusion.
\end{itemize}
Considering the above points, this leads to the following relaxation model 
\begin{align}
\begin{split}
\rho_t + \nabla \cdot (\rho u) &= 0,\\
(\rho \mathbf{u})_t + \nabla \cdot \left(\rho \mathbf{u}\otimes\mathbf{u} + \pi + \frac{1-M^2}{M^2} \psi\right) &= 0,\\
E_t + \nabla \cdot ( \mathbf{u} (E + M^2 \pi + (1 - M^2) \psi)) &= 0, \\
(\rho \pi)_t + \nabla \cdot (\rho \mathbf{u} \pi + a^2 \mathbf{u}) &= \frac{\rho}{\varepsilon}(p - \pi), \\
(\rho \mathbf{\hat{u}})_t + \nabla \cdot (\rho \mathbf{u} \otimes \mathbf{\hat{u}} + \frac{1}{M^2} \psi) &= \frac{\rho}{\varepsilon} (\mathbf{u} - \mathbf{\hat{u}}), \\
(\rho \psi)_t + \nabla \cdot (\rho \mathbf{u} \psi + a^2 \mathbf{\hat{u}}) &= \frac{\rho}{\varepsilon}(p - \psi).
\end{split}
\label{sys:full_relaxation_Euler_LM}
\end{align}
	The relaxation model given in \eqref{sys:full_relaxation_Euler_LM} differs from the one given in \cite{BerthonKlingenbergZenk2018} in the following points:
	\begin{enumerate}
	\item In the equation for $\mathbf{\hat{u}}$, we use $\frac{1}{M^2}$ instead of $\frac{1}{M^4}$ as proposed by the authors in \cite{BerthonKlingenbergZenk2018}.
	This is due to the upwind discretization used in \cite{BerthonKlingenbergZenk2018} which requires $\frac{1}{M^4}$ in order to have a Mach number independent diffusion of the numerical scheme. 
	Here instead we use centered differences in the implicit part to ensure the Mach number independent diffusion of the numerical scheme.
	\item We have simplified the model in the sense that we do not distinguish between the given Mach number $M$ and the local Mach number $M_{loc}$. 
	This is not a restriction in application, because the choice of $M$ is given by the application as illustrated in the numerical results.
	Especially for the Mach number dependent shock test case in Section 7.1.2 we directly compare our results to a scheme that uses the local Mach number $M_{loc}$.
	\end{enumerate}
The following lemma sums up some properties of system \eqref{sys:full_relaxation_Euler_LM}. 
The proof can be found in \cite{BerthonKlingenbergZenk2018}.
\begin{lemma}
	\label{lem:relax_sys_prob_orig}
	The relaxation system \eqref{sys:full_relaxation_Euler_LM} is hyperbolic and is a stable diffusive approximation of \eqref{sys:nondimEuler} under the Mach number independent sub-characteristic condition $a > \rho \sqrt{\partial_\rho p(\rho,e)}.$ 
	It has the following linearly degenerate eigenvalues 
	$$
	\lambda^u = \mathbf{u}, ~ \lambda^\pm = \mathbf{u} \pm \frac{a}{\rho}, ~ \lambda^\pm_M = \mathbf{u} \pm \frac{a}{M \rho},
	$$
	where $\lambda^u$ has multiplicity 4.
	For $M<1$, the eigenvalues have the ordering
	$$
	\lambda_M^- < \lambda^- < \lambda^u < \lambda^+ < \lambda^+_M.
	$$
\end{lemma}
In the case of $M=1$ the waves given by $\lambda^{\pm}$ and $\lambda_M^\pm$ collapse to the waves $\lambda^\pm$ which have then multiplicity 2 respectively.

To shorten notations we will refer to the original system \eqref{sys:nondimEuler} by 
\begin{align}
w_t + \nabla \cdot f(w) = 0,
\label{sys:nondim_short}
\end{align}
where $w = (\rho, \rho \mathbf{u}, E)^T$ denotes the physical variables and the flux function is given by $
f(w) = 
\begin{pmatrix}
\rho \mathbf{u} \\
\rho \mathbf{u} \otimes \mathbf{u} + \frac{p}{M^2}\mathbb{I} \\
\mathbf{u}(E + p)
\end{pmatrix}$.
The relaxation model \eqref{sys:nondimEuler} is given by 
\begin{align}
W_t + \nabla \cdot \mathcal{F}(W) = \frac{1}{\varepsilon}R(W),
\label{sys:relax_short}
\end{align}
where $W = (\rho, \rho \mathbf{u}, E, \rho \pi, \rho \mathbf{\hat{u}}, \rho \psi)$ denotes the state vector, $\mathcal{F}$ the flux function as defined in \eqref{sys:full_relaxation_Euler_LM} and $R$ the relaxation source term given by
\begin{align*}
R(W) = 
\begin{pmatrix}
0 \\ 0 \\ 0 \\ \rho(p - \pi) \\ \rho(\mathbf{u} - \mathbf{\hat{u}}) \\ \rho(p - \psi)
\end{pmatrix}.
\end{align*}
The relaxation equilibrium state is defined as
\begin{equation}
\label{eq:relax_equ}
W^{\text{eq}} = (\rho, \rho\mathbf{u},E,\rho p(\rho,e), \rho \mathbf{u},\rho p(\rho,e)).
\end{equation}
The connection between \eqref{sys:nondim_short} and \eqref{sys:relax_short} can be established, following \cite{ChenLevermoreLiu1994}, through the matrix $Q \in \mathbb{R}^{2+d\times 2(2+d)}$ defined by
\begin{equation}
Q=\begin{pmatrix}
\mathbb{I}_{2+d} & 0_{2+d}
\end{pmatrix}
\label{eq:Q}
\end{equation}
where $d$ is the dimension.
For all states $W$ it is satisfied $QR(W) = 0$ and the physical variables are then recovered by $w = QW$ and the fluxes are connected by $f(w) = Q\mathcal{F}(W^\text{eq}).$

\section{Time semi-discrete scheme}
\label{sec:time-semi}
As we have seen in Lemma \ref{lem:relax_sys_prob_orig}, the largest eigenvalue $|\lambda_M^\pm|$ of the relaxation model \eqref{sys:full_relaxation_Euler_LM} tends to infinity as $M$ goes to 0. 
Using a time explicit scheme results in a very restrictive CFL condition.
By using an IMEX approach as done in \cite{PareschiRusso2005,BoscarinoRussoScandurra2018}, we can avoid the Mach number dependence of the time step.

We rewrite the relaxation system \eqref{sys:full_relaxation_Euler_LM}/\eqref{sys:relax_short} in the following form:
\begin{equation}
	W_t + \nabla \cdot F(W) + \frac{1}{M^2}\nabla \cdot G(W) = \frac{1}{\varepsilon} R(W).
	\label{eq:decomp}
\end{equation}
In \eqref{eq:decomp}, we have split the flux $\mathcal{F}$ in \eqref{sys:relax_short} into a flux function $F$ which will contain the explicit terms and a flux function $G$ which will contain the terms treated implicitly.
For efficiency, we want to have as many explicit terms as possible as long as the eigenvalues of $F$ are independent of the Mach number. 
To avoid especially inverting a large non-linear system, we treat the non-linear advection terms explicitly.
This results in the following following flux functions
\begin{align}
F(W) = 
\begin{pmatrix} 
\rho \mathbf{u} \\
\rho \mathbf{u}\otimes\mathbf{u} + \pi \mathbbm{1} + \frac{1-M^2}{M^2} \psi \mathbbm{1} \\
( E + M^2 \pi + (1-M^2)\psi)\mathbf{u} \\
\rho \pi \mathbf{u} + a^2 \mathbf{u}\\
\rho \mathbf{u} \otimes \mathbf{\hat{u}} \\
\rho \psi \mathbf{u}
\end{pmatrix}
\text{ and }
G(W) = 
\begin{pmatrix}
0 \\ 0 \\ 0 \\ 0 \\ \psi\\ a^2 M^2\mathbf{\hat{u}}
\end{pmatrix}.
\label{eq:DefFG}
\end{align}
We see, that $F$ contains the flux $f$ of the Euler equations. \eqref{sys:nondimEuler} whereas $G$ and $R$ only act on the relaxation variables.
To obtain a time semi-discrete scheme we order the implicit and explicit steps as
 \begin{align}
 \text{Implicit: }	W_t + \frac{1}{M^2}\nabla \cdot G(W) & = 0, \label{eq:implicit} \\
 \text{Explicit: }~~~~~ 	W_t + \nabla \cdot F(W) & = 0, \label{eq:explicit}\\
 \text{Projection: }~~~~~~~~~~~~~~~~~~~~~	W_t &= \frac{1}{\varepsilon} R(W). \label{eq:projection}
 \end{align} 
 The relaxation source term in \eqref{eq:projection} is solved by projecting the variables onto the equilibrium manifold and thereby reaching the relaxation equilibrium state \eqref{eq:relax_equ}.
 The formal time semi-discrete scheme is then given by 
  \begin{align}
 W^{(1)} - W^{n,\text{eq}} + \frac{\Delta t}{M^2} \nabla \cdot G(W^{(1)}) &=0, \label{eq:t-semi-impl} \\
 W^{(2)} - W^{(1)} + \Delta t \nabla \cdot F(W^{(1)}) &=0, \label{eq:t-semi-expl} \\
 W^{n+1} = W^{(2),\text{eq}}\label{eq:t-semi-proj},
 \end{align}
 where we consider the data at time $t^n$ to be at relaxation equilibrium $W^{n,\text{eq}}$.
 First we solve the implicit equation \eqref{eq:update_impl} to gain $W^{(1)}$, followed by the explicit step where we calculate $W^{(2)}$. 
 To get the variables $W^{n+1}$ at the new time level, we project $W^{(2)}$ onto its equilibrium state. 
 This procedure results into a first order scheme in time. 

\section{Asymptotic properties}
\label{sec:AP}
We consider as continuous limit equations the incompressible Euler equations given by
\begin{align}
\begin{split}
\rho &= \text{const.}\\
u_t + u \cdot \nabla u + \nabla \Pi &= 0\\
\nabla \cdot u & = 0
\end{split}
\label{sys:incompr}
\end{align}
with a dynamical pressure described by $\Pi$.

Since the solution must converge to the solution of the incompressible system \eqref{sys:incompr} as $M \to 0$, see for example \cite{GuillardViozat1999,Dellacherie2010}, its asymptotic expansion under slipping or periodic boundary conditions must satisfy
\begin{align}
\rho &= \rho_0 + \mathcal{O}(M), &\rho_0 &= \text{const.} \label{eq:well_prep_data_rho}\\
u &= u_0 + \mathcal{O}(M), &\nabla \cdot u_0 &= 0 \label{eq:well_prep_data_u} \\
p &= p_0 + \mathcal{O}(M^2), &p_0 &= \text{const.}
\label{eq:well_prep_data_p}
\end{align} 
We refer to \eqref{eq:well_prep_data_rho}, \eqref{eq:well_prep_data_u}, \eqref{eq:well_prep_data_p} as well-prepared  data and summarize it in the following set
\begin{equation}
\label{set:AP}
\Omega_{wp} = \left\lbrace w \in \mathbb{R}^{d+2}, \nabla \rho_0 = 0, \nabla p_0 = 0, \nabla \cdot \mathbf{u}_0 = 0 \right\rbrace.
\end{equation} 
For simplicity we show the AP property for the time semi-implicit scheme. 
The same steps can be followed with the fully discretized scheme given in Section \ref{sec:fully_discrete}.

To show the AP property we will exploit some properties of the fast pressure $\psi^{(1)}$ obtained in the implicit step \eqref{eq:Implicit_part}.
  
\subsection{Asymptotic behaviour of $\psi^{(1)}$}
\label{sec:AP_psi}
Due to the sparse structure of $G$ defined in \eqref{eq:DefFG}, the implicit part reduces to solving only two coupled equations in the relaxation variables $\mathbf{\hat{u}}, \psi$ given by
\begin{align}
\begin{split}
(\rho\mathbf{\hat{u}})_t + \frac{1}{M^2} \nabla \psi &= 0, \\
(\rho\psi)_t + a^2 \nabla \cdot \mathbf{\hat{u}} &= 0, 
\end{split}
\label{eq:Implicit_part}
\end{align}
with the eigenvalues $\lambda_M = \pm \frac{a}{\rho M}$. 
As done in \cite{CordierDegondKumbaro2012}, we rewrite the coupled system \eqref{eq:Implicit_part} into one equation for $\psi$ starting from the time-semi-discrete scheme
\begin{align}
\frac{\rho^{(1)} - \rho^n}{\Delta t} &= 0, \label{eq:time-semi-rho}\\
\frac{(\rho \mathbf{\hat{u}})^{(1)} - (\rho \mathbf{\hat{u}})^{n} }{\Delta t} + \frac{1}{M^2}\nabla \psi^{(1)} &= 0, \label{eq:time-semi-uhat}\\
\frac{(\rho \psi)^{(1)} - (\rho \psi)^n}{\Delta t} + a^2 \nabla \cdot \mathbf{\hat{u}}^{(1)} &=0. \label{eq:time-semi-psi}
\end{align}
To emphasize that \eqref{eq:Implicit_part} also depends on the density, we have included the density update \eqref{eq:time-semi-rho} into the time-semi-discrete system.
From equation \eqref{eq:time-semi-rho} we see that $\rho^{(1)} = \rho^n$.
To simplify notation we define $\tau_i^n = \frac{1}{\rho_i^n}$.
Inserting \eqref{eq:time-semi-uhat} into \eqref{eq:time-semi-psi} we can reduce the implicit system to only one equation with an elliptic operator for $\psi$ given as
\begin{align}
\psi^{(1)}-\frac{\Delta t^2 a^2}{M^2} \tau^n \nabla \cdot (\tau^n \nabla \psi^{(1)}) &= \psi^n - \Delta t a^2 \tau^n \nabla \cdot u^n. 
\label{eq:elliptic_impl}
\end{align}
On the right hand side of \eqref{eq:elliptic_impl} we have already made use of the fact that $\hat{u}^n = u^n$ since we start from equilibrium data.
We will see that the correct scaling of $\psi^{(1)}$ with respect to the Mach number is important not only for showing the AP property of the scheme, but also for the positivity of density and internal energy as well as for the Mach number independent diffusion of the fully discretized scheme.

To prevent $\mathcal{O}(M)$ pressure perturbations at the boundaries which would destroy the well-prepared nature of the pressure, we require boundary conditions on $\psi$ which preserve the scaling of the pressure in time. 
For a computational domain $D$, we set
\begin{align}
\begin{rcases}
\psi_0^{(1)} &= p_0^n\\
\psi_1^{(1)} &= 0
\end{rcases}
\text{ on } \partial D.
\label{bc:pressure}
\end{align}

\begin{lemma}[Scaling of $\psi^{(1)}$]
	Let $w^n \in \Omega_{wp}$ be equipped with the boundary conditions \eqref{bc:pressure}.
	Then the Mach number expansion of $\psi^{(1)}$ after the first stage satisfies 
	\begin{equation}
	\label{eq:psi_expansion}
	\psi^{(1)} = p_0 + M^2 \psi_2 + \mathcal{O}(M^3),
	\end{equation} 
	where $p_0$ is constant.
	\label{lem:psi_scaling}
\end{lemma}
\begin{proof}
	Let us assume that the expansion of $\psi^{(1)}$ is given by 
	\begin{equation}
	\psi^{(1)} = \psi_0 + M \psi_1 + M^2 \psi_2 + \mathcal{O}(M^3).
	\label{eq:psi_scaling_assume}
	\end{equation}
	Since we start in $\Omega_{wp}$, we have well-prepared data as given in \eqref{eq:well_prep_data_rho}\eqref{eq:well_prep_data_u}\eqref{eq:well_prep_data_p}.
	We insert the well-prepared data $w^n$ and \eqref{eq:psi_scaling_assume} into the implicit update equation \eqref{eq:elliptic_impl}.
	Separating the $\mathcal{O}(M^{-2})$ terms, we find using the boundary conditions \eqref{bc:pressure}
	\begin{align*}
	\begin{cases}
		\Delta \psi_0 &= 0 \text{ in } D \\
		\psi_0 &= p_0 \text{ on } \partial D
	\end{cases}.
	\end{align*}
	This leads to $\psi_0 = p_0$ on $\overline{D}$. 
	Separating the $\mathcal{O}(M^{-1})$ terms and using that $\psi_0 = p_0 = \text{const}$, we find 
	\begin{align}
	\begin{cases}
	\Delta \psi_1 &= 0 \text{ in } D \\
	\psi_1 &= 0 \text{ on } \partial D
	\end{cases}
	\end{align}
	which leads to $\psi_1 = 0$ on $\overline{D}$.
	As a last step, we collect the $\mathcal{O}(M^{-1})$ terms and using that $\psi_0 = p_0$ as well as $\psi_1 = 0$ on $\overline{D}$.
	It is not necessary to impose special boundary conditions for $\psi_2$. 
	Thus we find
	\begin{align}
	\Delta \psi_2 &= 0 \text{ in } D.
	\label{eq:laplacianPsi2}
	\end{align}
\end{proof}

This shows that the fast pressure $\psi^{(1)}$ after the implicit step is still well-prepared. 

\subsection{Asymptotic preserving property}

We show that the time discretization of \eqref{sys:full_relaxation_Euler_LM} in the $M \to 0$ limit coincides with a time discretization of the incompressible Euler equations \eqref{sys:incompr}. 
We consider well-prepared data $w^n \in \Omega_{wb}$ and the Mach number expansion of $\psi^{(1)}$ from Lemma \ref{lem:psi_scaling}.
For the total energy defined in \eqref{eq:Etot}, we find the following Mach number expansion
\begin{equation}
	E = \rho_0 e_0 + M(\rho_1 e_0 + \rho_0 e_1 ) + M^2(|\mathbf{u}_0|^2 + \rho_2 e_0 + \rho_1 e_1 + \rho_0 e_2 ) + \mathcal{O}(M^3)
	\label{eq:MexpEtot}
\end{equation}
Inserting the Mach number expansions of $w^n, \psi^{(1)}$ and $E^n$ into the density, momentum and energy equation of \eqref{eq:t-semi-expl} and considering the $\mathcal{O}(M^0)$ order terms we have
\begin{align}
	\begin{split}
	\frac{\rho_0^{n+1} - \rho^n_0}{\Delta t} + \nabla \cdot \rho_0^n u_0^n &= 0, \\
	\frac{\rho_0^{n+1} u_0^{n+1}- \rho_0^n u_0^n}{\Delta t} + \nabla \cdot \left( \rho_0^n u_0^n \otimes u_0^n + \psi_2^{(1)}\right) &= 0, \\
	\frac{\rho^{n+1} e_0^{n+1} - \rho_0^n e_0^{n}}{\Delta t} + \nabla \cdot u_0^n(\rho_0^n e_0^n + \psi_0^{(1)}) &= 0, \\
	\frac{\rho_0^{n+1} e_0^{n+1} - \rho_0^n e_0^{n}}{\Delta t} + \nabla \cdot u_0^n(\rho_0^n e_0^n + \psi_0^{(1)}) &= 0
	\label{eq:ZeroOrder_exp}
	\end{split}
\end{align}
Let us assume, the pressure at time $t^{n+1}$ has the Mach number expansion $p^{n+1} = p_0^{n+1} + Mp_1^{n+1} + \mathcal{O}(M^2)$.
Note that we have from Lemma \ref{lem:psi_scaling} that $\psi_0^{(1)} = p_0$ and thus $\nabla \psi_0^{(1)} = \nabla p_0 = 0$.
Since $w^n \in \Omega_{wp}$ we have $\nabla \rho_0^n = 0$ and $\nabla \cdot u_0^n = 0$.
	Further we use $p_0 = (\gamma -1) \rho_0 e_0$. 
	Equipped with that we can simplify \eqref{eq:ZeroOrder_exp} to 
\begin{align}
\frac{\rho_0^{n+1} - \rho^n_0}{\Delta t} &= 0, \label{eq:incompr_disc_rho}\\
\frac{u_0^{n+1}- u_0^n}{\Delta t} + u_0^n \cdot \nabla u_0^n + \frac{\nabla \psi_2^{(1)}}{\rho_0^n} &= 0, \label{eq:incompr_disc_u}\\
\frac{p_0^{n+1} - p_0^{n}}{\Delta t} &= 0.
\label{eq:incompr_disc_p}
\end{align}
Especially from equations \eqref{eq:incompr_disc_rho} and \eqref{eq:incompr_disc_p} we see that $\rho_0^{n+1}$ and $p_0^{n+1}$ are constants. 
Looking at the $\mathcal{O}(M^1)$ terms we have from the energy equation
\begin{align}
	p_1^{n+1} + \Delta t \nabla \cdot u_1^n = 0.
\end{align}
This means the density and pressure at time $t^{n+1}$ are well-prepared up to $\mathcal{O}(\Delta t)$ perturbation as in \eqref{eq:well_prep_data_rho}, \eqref{eq:well_prep_data_p}.
To be consistent with a time discretization of the incompressible Euler equations  \eqref{sys:incompr} the divergence of $\mathbf{u}_0$ at time $t^{n+1}$ defined by $\nabla \cdot \mathbf{u}_0^{n+1}$ has to be at least of order $\Delta t$. 
To show this, we apply the divergence on the velocity update \eqref{eq:incompr_disc_u} which gives
\begin{align}
	\nabla \cdot \mathbf{u}_0^{n+1} = \nabla \cdot \mathbf{u}_0^n + \Delta t \nabla \cdot (u_0^n \cdot \nabla u_0^n + \frac{\Delta \psi_2^{(1)}}{\rho^n_0}).
	\label{eq:divu_1}
\end{align}
In the proof of Lemma \ref{lem:psi_scaling}, we have shown that $\Delta \psi_2^{(1)} = 0$ on $\partial D$, see \eqref{eq:laplacianPsi2}.
Using \eqref{eq:laplacianPsi2} together with $\nabla \cdot u^n_0 = 0$, we can simplify \eqref{eq:divu_1} to 
\begin{align*}
\nabla \cdot \mathbf{u}_0^{n+1} = \Delta t \nabla \cdot (u_0^n \cdot \nabla u_0^n) = \mathcal{O}(\Delta t).
\end{align*}
In summary, we have shown the following theorem. 
\begin{theorem}[AP property]
	Let $w^n \in \Omega_{wp}$. Then under the boundary conditions \eqref{bc:pressure} the scheme \eqref{eq:t-semi-impl}, \eqref{eq:t-semi-expl}, \eqref{eq:t-semi-proj} is asymptotic preserving when $M$ tends to $0$, 
	in the sense that if $w^n \in \Omega_{wp}$ then it is $w^{n+1}\in \Omega_{wp}$ and in the limit $M\to 0$ the time-semi-discrete scheme is a consistent discretization of the incompressible Euler equations \eqref{sys:incompr}.
	
\end{theorem}

\section{Derivation of the fully discrete scheme}
\label{sec:fully_discrete}
For simplicity, we develop the fully discretized scheme in one space dimension, but it can be straightforwardly extended to $d$ dimensions.
In the implicit update \eqref{eq:elliptic_impl}, the space derivatives read
$$
\nabla \cdot (\tau \nabla \psi) = \partial_{x_1}(\tau \partial_{x_1}\psi) + \dots +\partial_{x_d}(\tau \partial_{x_d}\psi) \text{ and } \nabla \cdot u = \partial_{x_1} u_1 + \dots +\partial_{x_d} u_d
$$
and in the explicit part we apply dimensional splitting.

In the following we use a cartesian grid on a computational domain $D$ devided in $N$ cells $C_i = (x_{i-1/2},x_{i+1/2})$ of step size $\Delta x$. 
We use a standard finite volume setting, where we define at time $t^n$ the piecewise constant functions 
$$
	w(x,t^n) = w_i^n, \text{ for } x \in C_i.
$$ 
Using this notation, we apply centered differences for the implicit update \eqref{eq:elliptic_impl} and have
\begin{align}
\label{eq:update_impl}
\begin{split}
\psi_i^{(1)} &- \frac{\Delta t^2}{\Delta x^2}\frac{a^2}{M^2} \tau_i^n \left(\tau_{i-1/2}^n \psi_{i-1}^{(1)} - (\tau_{i-1/2}^n + \tau_{i+1/2}^n) \psi_i^{(1)} + \tau_{i+1/2}^n \psi_{i+1}^{(1)}\right) = \\
\psi_i^n & - \frac{\Delta t}{2\Delta x} a^2 \tau_i^n\left(u_{i+1}^n - u_{i-1}^n\right),
\end{split}
\end{align}
where $\tau_{i+1/2} = \frac{1}{2}\left(\tau_{i+1} + \tau_{i}\right)$. 

For the explicit part, we will use a Godunov type finite volume scheme following \cite{HartenLaxVanLeer1983} which we will describe in the following section. 

\subsection{Godunov type finite volume scheme} 

In the explicit step we consider the following equations as defined in \eqref{eq:DefFG}, \eqref{eq:t-semi-expl}
\begin{align}
\begin{split}
	\partial_t \rho + \partial_{x} \rho u &= 0, \\
	\partial_t \rho u + \partial_{x} (\rho u^2 + \pi + \frac{1-M^2}{M^2}\psi) &= 0,\\
	\partial_t E + \partial_{x}((E + M^2 \pi + (1-M^2)\psi) u) &= 0, \\
	\partial_t \rho \pi + \partial_{x}((\rho\pi + a^2)u) &= 0, \\
	\partial_t \rho\hat{u} + \partial_{x}(\rho u \hat{u}) &= 0, \\
	\partial_t \rho\psi + \partial_{x}(\rho \psi u) &= 0.
\end{split}
\label{sys:expl}
\end{align}

In the next result the structure of system \eqref{sys:expl} is summarized.
\begin{lemma}
	\label{lem:EW}
	System \eqref{sys:expl} admits the linear degenerate eigenvalues $\lambda^\pm =u \pm \frac{a}{\rho}$ and $\lambda^u = u$, where the eigenvalue $\lambda^u$ has multiplicity 4. 
	The relaxation parameter $a$ is independent of the Mach number as well as all eigenvalues. 
	The Riemann Invariants with respect to $\lambda^u$ are
	\begin{equation}
	\label{eq:RI_u}
	I_1^u = u,~ I_2^u = M^2 \pi + (1 - M^2) \psi
	\end{equation}
	and with respect to $\lambda^\pm$ 
	\begin{align}
	\label{eq:RI_pm}
	\begin{split}
	I_1^\pm &= u \pm \frac{a}{\rho},~I_2^\pm = \pi \mp a u, \\
	I^\pm_3 &= e - \frac{M^2}{2a^2}\pi^2 - \frac{1-M^2}{a^2} \pi \psi,\\ 
	I_{4}^\pm &= \hat{u},~~~I^\pm_5 = \psi.
	\end{split}
	\end{align}
\end{lemma}
\begin{proof}
	We rewrite the equations \eqref{sys:expl} using primitive variables $V = (\rho, u, e, \pi, \hat{u}, \psi)$ in non-conservative form
	\begin{equation}
		V_t + \mathcal{A}(V)V_x = 0,
		\label{eq:nonconsform}
	\end{equation}
	where the matrix $\mathcal{A}(V)$ is given by
	\begin{align}
		\mathcal{A}(V) = 
		\begin{pmatrix}
		u& \rho& 0& 0& 0& 0 \\
		0& u& 0& \frac{1}{\rho}& 0& \frac{1-M^2}{M^2} \\
		0& \frac{M^2 \pi + (1-M^2)\psi}{\rho}& u & 0& 0&0\\
		0& \frac{a^2}{\rho}& 0 & u & 0 & 0 \\
		0& 0& 0& 0&u& 0 \\
		0& 0& 0& 0&0& u \\
		\end{pmatrix}.
	\end{align}
	It is easy to check that $\lambda^u, \lambda^\pm$ are eigenvalues of $\mathcal{A}(V)$.
	Associated to the eigenvalues, we find the respective eigenvectors
	\begin{align*}
		r_1^u = 
		\begin{pmatrix}
		0 \\ 0 \\ 0 \\ 1-\frac{1}{M^2} \\ 0 \\ 1
		\end{pmatrix}
		r_2^u = 
		\begin{pmatrix}
		0 \\ 0 \\ 0 \\ 0 \\ 1 \\ 0
		\end{pmatrix}
		r_3^u = 
		\begin{pmatrix}
		0 \\ 0 \\ 1 \\ 0 \\ 0 \\ 0 \\
		\end{pmatrix}
		r_4^u = 
		\begin{pmatrix}
		1 \\ 0\\ 0\\0 \\0 \\0 \\
		\end{pmatrix}
		r^\pm = 
		\begin{pmatrix}
		\frac{\rho ^2}{a^2}\\
		\pm \frac{1}{a}\\
		\frac{M^2\pi+(1-M^2)\psi }{a^2}\\
		1\\
		0\\
		0\\
		\end{pmatrix}
	\end{align*}
	We see that $\partial_V\lambda^u \cdot r_i = 0$, $\forall i=1,\dots,4$ and $\partial_V \lambda^\pm \cdot r^\pm = 0$.
	Thus all eigenvalues are linearly degenerate.
	A scalar function $I(V)$ is a Riemann invariant if for all eigenvectors $r$ associated to an eigenvalue $\lambda$, $\partial_V I(V) \cdot r = 0$ holds.
	It is straightforward to check, that \eqref{eq:RI_u} and \eqref{eq:RI_pm} are Riemann invariants.
	Since Riemann invariants are invariant under change of variables, the Riemann invariants of \eqref{eq:nonconsform} are the same as for the equations in conservation form \eqref{sys:expl}. 
	For more details, see \cite{Bouchut2004,BerthonKlingenbergZenk2018}.
\end{proof}

We will follow the theory of Harten, Lax and van Leer \cite{HartenLaxVanLeer1983} for deriving an approximate Riemann solver $W_\mathcal{RS}\left(\frac{x}{t}; W_L^{(1)}, W_R^{(1)}\right)$ based on the states $W^{(1)}$ after the implicit step.
Due to the linear-degeneracy from Lemma \ref{lem:EW}, the structure of the approximate Riemann solver, as displayed in Figure \ref{fig:Riemann_solution}, is given as follows
\begin{align}
W_\mathcal{RS}\left(\frac{x}{t}; W_L^{(1)}, W_R^{(1)}\right) = 
\begin{cases}
W_L^{(1)} & \frac{x}{t} < \lambda^- ,\\
W^*_L & \lambda^- < \frac{x}{t} < \lambda^u ,\\
W^*_R & \lambda^u < \frac{x}{t} < \lambda^+ ,\\
W_R^{(1)} & \lambda^+ < \frac{x}{t}.
\end{cases} 
\label{eq:Sol_RP}
\end{align}
\begin{figure}[htbp]
	\centering
	\begin{tikzpicture}[scale=2.]
	\begin{scope}[every node/.style={scale=1.}]
	\draw
	(-0.5,0)-- (2.5,0);
	\draw (0,0.85) coordinate (l1) node[above] {$u - a/\rho$} -- (1,0) coordinate (x0) node[below, yshift=-0.5ex] {$x=0$} ;
	\draw (2,0.85) coordinate (l3) node[above right] {$u + a/\rho$} -- (1,0);
	\draw (1.15,0.85) coordinate (l2) node[above,yshift=0.5ex] {$u$} -- (1,0);
	\draw (0.25,0.4) coordinate (WL) node[left] {$W_L$};
	\draw (1.75,0.4) coordinate (WR) node[right] {$W_R$};
	\draw (0.7,0.5) coordinate (WSL) node[above] {$W^*_L$};
	\draw (1.4,0.5) coordinate (WSR) node[above] {$W^*_R$};
	\end{scope}
	\end{tikzpicture}
	\label{fig:Riemann_solution}
	\caption{Structure of the Riemann solution.}
\end{figure}
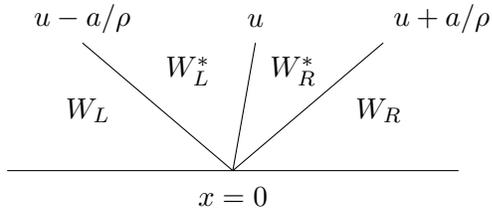
To compute the intermediate states $W_{L,R}^{\ast}$, we use the Riemann invariants as given in Lemma \ref{lem:EW}. 
Note that since the eigenvalues $\lambda^\pm$ have multiplicity 1, we get the expected 5 Riemann invariants.
This does not hold in general for eigenvalues with multiplicity larger than 1. 
Nevertheless, the invariants \eqref{eq:RI_u} and \eqref{eq:RI_pm} give enough relations to determine the solution to a Riemann problem for \eqref{sys:expl} as shown in the following lemma. 
\begin{lemma}
	\label{lem:Sol_RP}
	Consider an initial value problem with initial data $W = W^{(1)}$ given by 
	\begin{align}
	W_0(x) = 
	\begin{cases}
	W_L & x < 0, \\
	W_R & x > 0.
	\end{cases}
	\label{eq:RP}
	\end{align}
	Then the solution consists of four constant states separated by contact discontinuities with the structure given in \eqref{eq:Sol_RP}.
	The solution for the states $W^\ast_{L/R}$ is given by 
	\begin{align}
	\begin{split}
	\frac{1}{\rho^\ast_L} &= \frac{1}{\rho_L} + \frac{1}{a}(u^\ast - u_{L}),\\
	\frac{1}{\rho^\ast_R} &= \frac{1}{\rho_R} + \frac{1}{a}(u_{R}- u^\ast),\\
	u^\ast &= u^\ast_L = u_R^\ast = \frac{1}{2}(u_{L} + u_{R}) + \frac{1}{2a}\left[(\pi_L - \pi_R) + \frac{1-M^2}{M^2} (\psi_L - \psi_R)\right], \\ 
	\pi_{L}^\ast &= \frac{1}{2}(\pi_L + \pi_R) + \frac{a}{2}(u_{L} - u_{R}) - \frac{1-M^2}{M^2}\frac{1}{2}(\psi_L - \psi_R), \\
	\pi_{R}^\ast &= \frac{1}{2}(\pi_L + \pi_R) + \frac{a}{2}(u_{L} - u_{R}) + \frac{1-M^2}{M^2}\frac{1}{2}(\psi_L - \psi_R), \\
	e_L^\ast &= e_L - \frac{1}{2a^2} \left[(\pi_L^2 - (\pi_L^\ast)^2) + (1-M^2)(\pi_L - \pi_L^\ast)\psi_L\right], \\
	e_R^\ast &= e_R - \frac{1}{2a^2} \left[(\pi_R^2 - (\pi_R^\ast)^2) + (1-M^2)(\pi_R - \pi_R^\ast)\psi_R\right], \\
	\psi_{L,R}^\ast &= \psi_{L,R}, \\
	\hat{u}^\ast_{L,R} &= \hat{u}_{L,R}.
	\end{split}
	\label{eq:intermStates}
	\end{align}
\end{lemma}
\begin{proof}
	Since $\frac{a}{\rho} > 0$, we have the following order of the eigenvalues
	$u - \frac{a}{\rho} < u < u + \frac{a}{\rho}$.
	The solution structure follows directly from the linear degeneracy of the eigenvalues given in Lemma \ref{lem:EW} and the ordering of the eigenvalues. 
	To derive the solution for the intermediate states $W^\ast_{L,R}$ one uses the Riemann invariants given in \eqref{eq:RI_u} and \eqref{eq:RI_pm} and solves the resulting system of equations.
\end{proof}

 Given the solution of the Riemann problem \eqref{eq:Sol_RP}, we define the numerical fluxes as follows
 \begin{align}
 	F_{i+1/2} = 
 	\begin{cases}
 		F(W_i^{(1)}) &\lambda^- > 0 \\
 		F(W_i^{\ast,(1)}) & \lambda^u > 0 > \lambda^- \\
 		F(W_{i+1}^{\ast,(1)}) & \lambda^+ > 0 > \lambda^u \\
 		F(W_{i+1}^{(1)}) & \lambda^+ < 0
 	\end{cases}
 	\label{def:F_interface}
 \end{align}
 where the superscript $(1)$ emphasizes that the states after the implicit step are used.
 To avoid interactions between the approximate Riemann solvers at the interfaces $x_{i+1/2}$, we have a CFL restriction on the time step of 
 \begin{equation}
 \Delta t \leq \frac{1}{2} \frac{\Delta x}{ \underset{i}{\max}|u_i \pm a/\rho_i|}
 \label{eq:cfl}
 \end{equation} 
 which is independent of the Mach number.
 This leads to the following update of the explicit part
\begin{equation}
W^{(2)}_i = W^{(1)}_i - \frac{\Delta t}{\Delta x} (F_{i+1/2} - F_{i-1/2}).
\label{eq:update_explicit}
\end{equation}  
 
 At this point, we can make some remarks on the numerical scheme as it is given by \eqref{eq:update_impl},\eqref{eq:update_explicit} and \eqref{eq:t-semi-proj}.
 \begin{remark}
 	\label{rem:expl_upd}
 	The flux function $G$ in \eqref{eq:DefFG} in the implicit part and the relaxation source term $R$ only act on the relaxation variables $(\pi, \mathbf{\hat{u}},\psi)$. Thus, we can immediately give the update of the physical variables $w$ as 
	\begin{equation}
		w^{n+1}_i = w_i^n - \frac{\Delta t}{\Delta x}\left(f\left(QW_\mathcal{RS}\left(0;W_i^{(1)}, W_{i+1}^{(1)}\right)\right) - f\left(QW_\mathcal{RS}\left(0;W_{i-1}^{(1)}, W_{i}^{(1)}\right)\right)\right),
		\label{eq:first_pyhs_upd}
	\end{equation}
	where the matrix $Q$ is the projection on the first $2+d$ components as defined in \eqref{eq:Q}.
 \end{remark}
\begin{remark}\label{rem:uhat_dep}
	We see from \eqref{eq:update_impl} and \eqref{eq:intermStates}, that the updates in the implicit and explicit step are independent of $\mathbf{\hat{u}}$. 
	Therefore it is not necessary to explicitly compute $\mathbf{\hat{u}}$ in the scheme thus reducing the computational costs. 
\end{remark}
\begin{remark}
	\label{rem:RS_M=1}
	For the limit case $M =1$ the relaxation model \eqref{sys:full_relaxation_Euler_LM} reduces to a Suliciu relaxation model for the compressible Euler equations since all $\psi$-terms are being cancelled and the eigenvalues $\lambda^\pm$ and $\lambda_M^\pm$ collapse.
	Analogously, in the scheme the implicit step becomes redundant, since in the Riemann solution \eqref{eq:intermStates} the $\psi$-terms are cancelled, too. Thus the scheme reduces to a Godunov-type approximate Riemann solver based explicit scheme for the compressible Euler equations as can be found in \cite{Bouchut2004}.
	
	For $M>1$ the order of the eigenvalues changes to $\lambda^- < \lambda_M^- < \lambda^u <\lambda_M^+ < \lambda^+$. The scheme is still stable because the CFL condition \eqref{eq:cfl} comes from the explicit eigenvalues which are now the largest ones.
\end{remark}

\subsection{Positivity of density and internal energy}
For physical applications it is necessary that the density and internal energy be positive. 
We define the physical admissible states associated with the Euler equations \eqref{sys:nondimEuler} as
\begin{equation}
\Omega_{phy} = \left\lbrace w \in \mathbb{R}^{d+2}, \rho > 0, e > 0 \right\rbrace.
\end{equation}
The property of our scheme to preserve the domain $\Omega_{phy}$ is linked to how the fluxes are calculated. 
In our case it is essential that the density and internal energy of the Riemann solution \eqref{eq:Sol_RP},\eqref{eq:intermStates} are positive.
	This is shown in the following lemma.	
\begin{lemma}
	\label{lem:RS_pos1}
	Let the initial data of the Riemann problem \eqref{eq:RP} be given as $W_{L,R}^{(1)} \in \Omega_{phy} \cap \Omega_{wp}$ satisfying the boundary conditions \eqref{bc:pressure}. 
	Then there is a relaxation parameter large enough but independent of $M$ such that $W_\mathcal{RS}(\frac{x}{t},W_L^{(1)},W_R^{(1)}) \in \Omega_{phy}$.
\end{lemma}
\begin{proof}
	Since the proof only concerns data after the implicit step, we will drop the superscript $(1)$.
	We have to prove, that $\rho_{L,R}^\ast >0 $ and $e_{L,R}^\ast > 0$.
	From the ordering of eigenvalues $u_L - \frac{a}{\rho_L} < u^\ast$ and the formula for $\rho_L^{\ast}$ from \eqref{eq:intermStates}, we get 
	$$
	\frac{1}{\rho_L^\ast} = \frac{1}{\rho_L} + \frac{u^\ast - u_L}{a} \geq \frac{1}{\rho_L} - \frac{1}{\rho_L} = 0. 
	$$
	Analogously, we find $\rho_R^\ast \geq 0$.
	For the internal energy $e_L^\ast$, we insert the definition of $\pi_L^\ast$ from \eqref{eq:intermStates} into $e_L^\ast$ to obtain a formula depending only on the left and right states $W_{L,R}$
	\begin{align*}
	\begin{split}
	e_L^\ast &= e_L + \frac{1}{8}(u_L - u_R)^2+ \\
	&\frac{1}{2 a^2} \left(-\pi_L^2 + \frac{1}{4}\left(\pi_L + \pi_R + \frac{1-M^2}{M^2}(\psi_R - \psi_L)\right)^2 \right.\\
	& \hspace{1cm}\left.+ \frac{1}{2}\psi_L(1-M^2) \left(\pi_R - \pi_L + \frac{1-M^2}{M^2}(\psi_R - \psi_L)\right)\right) \\
	&-\frac{1}{4a}\left(u_L - u_R\right)\left(\pi_L + \pi_R + \frac{1-M^2}{M^2}(\psi_R - \psi_L) + (1-M^2)\psi_L\right).
	\end{split}
	\end{align*}
	From Lemma \ref{lem:psi_scaling}, we know $\psi_{L,R} = p_0 + \mathcal{O}(M^2)$.
	The difference  $\psi_R - \psi_L = \mathcal{O}(M^2)$ cancels with the $1/M^2$ which leads to $e_L^\ast = \mathcal{O}(1)$. 
	All possibly negative terms in $e_L^\ast$ can then be controlled by the relaxation parameter $a>0$ independent of the Mach number.
	The same argument holds for $e_R^\ast$.
\end{proof} 

The positivity property for the first order scheme is given in the next result.
\begin{theorem}[Positivity property 1]
	\label{theo:Pos_first}
		Let the initial state be given as $$w_i^n \in \Omega = \Omega_{phy} \cap \Omega_{wp}$$ in $d$ dimensions satisfying the boundary condition described in \eqref{bc:pressure}. 
		Then under the Mach number independent CFL condition $\frac{\Delta t}{\Delta x} \underset{i}{\max}|\lambda^\pm(w_i^n)| < \frac{1}{2d}$, the numerical scheme defined by \eqref{eq:update_impl},\eqref{eq:first_pyhs_upd} preserves the positivity of density and internal energy, that is 
		$$
		w_{i}^{n+1} \in \Omega_{phy}.
		$$
\end{theorem}
\begin{proof}
	Due to the construction of the numerical scheme, the update of the physical variables is only done through the explicit step \eqref{eq:first_pyhs_upd}. 
	Therefore we can adopt the proof of Theorem 3 in \cite{ThomannZenkKlingenberg2019} using the positivity of the Riemann solver from Lemma \ref{lem:RS_pos1}.
	The key element is, that we can write the update $w^{n+1}_i$ in one dimension as a convex combination of Riemann solvers which satisfy $W_{\mathcal{RS}} \in \Omega_{phy}$ according to Lemma \ref{lem:RS_pos1}:
	\begin{align}
	\label{eq:update_RS}
	\begin{split}
		w_i^{n+1} =
		\frac{1}{\Delta x} \left(\int_{x_{i-\frac{1}{2}}}^{x_{i}} QW_{\mathcal{RS}}\left(\frac{x}{t^{n+1}},W_{i-1},W_i\right)dx
		+ \int_{x_{i}}^{x_{i+\frac{1}{2}}}Q W_{\mathcal{RS}}\left(\frac{x}{t^{n+1}},W_{i},W_{i+1}\right) dx \right)
	\end{split}
	\end{align}
	Since we are using dimensional splitting, the update in $d$ dimensions can be written as a sum of updates as given in \eqref{eq:update_RS} in each dimension and due to convexity we have $w_i^{n+1} \in \Omega_{phy}$. More details can be found for example in \cite{ThomannZenkKlingenberg2019}.
\end{proof}
	
	\subsection{Mach number independent diffusion}
	Although we are using a Godunov type upwind scheme in the explicit part, our scheme does not suffer from an excessive numerical diffusion as $M$ tends to 0. 
	As we will show in the following this is due to the well-prepared implicitly treated fast pressure $\psi^{(1)}$. 
	In order to do so, we investigate the numerical diffusion vector $\mathcal{D}\in \mathbb{R}^{2+d}$ defined by 
	\begin{align}
	\label{eq:diff_fast}
		\mathcal{D} = \frac{f(w_{i}^n) + f(w_{i+1}^n)}{2} - f\left(QW_\mathcal{RS}\left(W_i^{(1)}, W_{i+1}^{(1)}\right)\right),
	\end{align}
	where the Matrix $Q$ is defined in \eqref{eq:Q}.
	Given well-prepared initial data $w^n_i \in \Omega_{wp}$, we have the following Mach number expansion for the physical variables as given in \eqref{eq:well_prep_data_rho},\eqref{eq:well_prep_data_u} and \eqref{eq:well_prep_data_p}
	\begin{align}
	\begin{array}{clccl}
		\rho_i &= ~~\rho_0 + \mathcal{O}(M), &\hspace{2cm}&\rho_{i+1} &= ~~\rho_0 + \mathcal{O}(M) \\
		u_i &= ~~u_{0,i} + \mathcal{O}(M), &&u_{i+1} &= ~~u_{0,{i+1}} + \mathcal{O}(M) \\
		e_i &= ~~e_0 + \mathcal{O}(M), &&e_{i+1} &= ~~e_0 + \mathcal{O}(M) \\ 
		\pi_i &= ~~p_0 + \mathcal{O}(M^2), &&\pi_{i+1} &= ~~p_0 + \mathcal{O}(M^2)
	\end{array}
	\label{eq:scaling_diff_n}
	\end{align}
	From Lemma \ref{lem:psi_scaling}, we have for $\psi$ after the implicit step 
	\begin{align}
		\psi_i &= p_0 + \mathcal{O}(M^2), & \psi_{i+1} &= p_0 + \mathcal{O}(M^2).
		\label{eq:scaling_diff_impl}
	\end{align}
	The Mach number expansion of the states $W^{(1)}_{i},W^{(1)}_{i+1}$ used in the Riemann solver $W_{\mathcal{RS}}$ is composed of the expansions \eqref{eq:scaling_diff_n} and \eqref{eq:scaling_diff_impl}.
	Inserting them into the formulas of the intermediate states \eqref{eq:intermStates} of the Riemann solution \eqref{eq:Sol_RP}, we have the following scaling of $W_i^\ast, W_{i+1}^\ast$ with respect to the Mach number
	\begin{align}
	\begin{array}{clccl}
		\tau_{i}^\ast &= ~~\tau_0 + \mathcal{O}(1), &\hspace{0.5cm}& \tau_{i+1}^\ast &=~~ \tau_0 + \mathcal{O}(1)  \\
		u_{i+1/2}^\ast &=~~ \frac{u_{0,i} + u_{0,i+1}}{2} + \mathcal{O}(1) = u_{0,i+1/2} +  \mathcal{O}(1) && &\\
		e_{i}^\ast &=~~ e_0 + \mathcal{O}(1),&&e_{i+1}^\ast &= ~~e_0 + \mathcal{O}(1)\\
		\pi_{i}^\ast &=~~ p_0 + \mathcal{O}(1),&& \pi_{i+1}^\ast &=~~ p_0 + \mathcal{O}(1),\\
		\psi_{i}^\ast &=~~ p_0 + \mathcal{O}(M^2),&& \psi_{i+1}^\ast &=~~ p_0 + \mathcal{O}(M^2).
	\end{array}
	\label{eq:expand_M_interm}
	\end{align}
	From \eqref{eq:expand_M_interm} it is evident that the lowest order of $M$ in the intermediate states is $\mathcal{O}(M^0)$.
	Inserting \eqref{eq:expand_M_interm} in the interface flux gives
	\begin{align}
	\label{eq:diff_f_interm}
	f\left(QW_\mathcal{RS}\left(W_i^{(1)}, W_{i+1}^{(1)}\right)\right) =
	\begin{pmatrix}
	\rho_0 u_{0,i+1/2} + \mathcal{O}(1)\\
	\rho_0 u_{0,i+1/2}^2 + \frac{p_0}{M^2}+\mathcal{O}(1)\\
	u_{0,i+1/2} (E_0 + p_0) + \mathcal{O}(1)	
	\end{pmatrix}.
	\end{align}
	Therefore, using \eqref{eq:diff_f_interm} in \eqref{eq:diff_fast}, the diffusion vector is given by 
	\begin{align*}
	\mathcal{D} = 
	\begin{pmatrix}
	\mathcal{O}(1)\\
	\mathcal{O}(1)\\
	\mathcal{O}(1)
	\end{pmatrix}.
	\end{align*}
	This shows that the diffusion introduced by the Riemann solver does not suffer from a $\mathcal{O}(M^{-1})$ dependent diffusion in the momentum equation.

\section{Second order extension}
\label{sec:Second}
In this section we extend the first order scheme to second order accuracy. 
We seek a natural extension of the first order scheme that preserves the positivity property as shown in Theorem \ref{theo:Pos_first}.  
\subsection{Second order time integration scheme}
For the second order extension, we use a modified two stage time integration as can be found in \cite{Berthon2005}. 
It is a convex combination of first order steps and is given by
\begin{align}
\begin{split}
\overline{W}^{(1)} =& W^{\text{eq},n} +\frac{ \Delta t}{M^2} \nabla \cdot G(\overline{W}^{(1)}),\\
\overline{W}^{(2)} =&\overline{W}^{(1)} + \Delta t \nabla \cdot F(\overline{W}^{(1)}),\\
\overline{\overline{W}}^{(1)} =& \overline{W}^{\text{eq},(2)} + \frac{\Delta t}{M^2} \nabla \cdot G(\overline{\overline{W}}^{(1)}),\\
\overline{\overline{W}}^{(2)} =& \overline{\overline{W}}^{(1)} + \Delta t \nabla \cdot  F(\overline{\overline{W}}^{(1)}),\\
W^{n+1} =& \frac{1}{2} \overline{\overline{W}}^{\text{eq},(2)} + \frac{1}{2} W^{\text{eq},n}.
\end{split}
\label{sys:Time_int2}
\end{align}
The relaxation equilibrium states $W^{\text{eq},n}$ and $W^{\text{eq},(2)}$ are defined as in \eqref{eq:relax_equ}.
As in the first order case, we can rewrite the integration scheme \eqref{sys:Time_int2} in terms of the update for the physical variables given in Remark \ref{rem:expl_upd}. 
The first order time update \eqref{eq:first_pyhs_upd} can be written as 
\begin{equation*}
w^{n+1} = w^n - \Delta t \nabla \cdot f(QW^{(1)}),
\end{equation*}
where $Q$ is defined in \eqref{eq:Q}.
Based on this we can reformulate \eqref{sys:Time_int2} in terms of updates of the physical variables as
\begin{align}
\begin{split}
\overline{w} =& w^n - \Delta t \nabla \cdot f(Q\overline{W}^{(1)}) \\
\overline{\overline{w}} =& \overline{w} -\Delta t \nabla \cdot f(Q\overline{\overline{W}}^{(1)})\\
w^{n+1} &= \frac{1}{2}w^n + \frac{1}{2}\overline{\overline{w}}.
\end{split}
\label{sys:Time_int2_phys}
\end{align}
where we obtain $W^{(1)}$ and $\overline{W}^{(1)}$ from the implicit step \eqref{eq:t-semi-impl} using the update formula \eqref{eq:update_impl} to compute $\psi^{(1)}$ and $\overline{\psi}^{(1)}$.
That the time integration scheme is second order accurate is numerically validated in Section \ref{sec:NumRes}.
The time integration scheme \eqref{sys:Time_int2_phys} can be extended to variable step sizes $\Delta t_1, \Delta t_2$ for each stage respectively as given in \cite{Berthon2005}.
This has the advantage that the CFL criterion can be met at each stage independently.
It is given by
\begin{align}
\begin{split}
\overline{w} =& w^n - \Delta t_1 \nabla \cdot f(Q\overline{W}^{(1)}) \\
\overline{\overline{w}} =& \overline{w} -\Delta t_2 \nabla \cdot f(Q\overline{\overline{W}}^{(1)})\\
w^{n+1} &= \left(1 -\frac{2\Delta t_1 \Delta t_2}{\Delta t_1 + \Delta t_2}\right)w^n + \frac{2\Delta t_1 \Delta t_2}{\Delta t_1 + \Delta t_2}\overline{\overline{w}}.
\end{split}
\label{sys:Time_int2_phys_var_step}
\end{align}

\subsection{Second order space reconstruction} 
Following Remark \ref{rem:expl_upd}, we focus on the update of the physical variables in the explicit step. 
As it is standard in the finite volume setting, we apply a reconstruction to get a higher accuracy for the interface values $w_{i+1/2}$.
To get second order, we consider piecewise linear functions in the conserved variables $w = (\rho, \rho \mathbf{u}, E)$ at time level $t^n$ and in $\psi^{(1)}$. 
As we are working on a Cartesian grid, we reconstruct along each dimension separately. 
In one dimension the linear function in $(w_i,x_i)$ on $(x_{i-1/2},x_{i1/2})$ is given by 
\begin{equation*}
	w^n(x) = w_i^n + \sigma_i(x - x_i),
\end{equation*}
where $\sigma_i = (\sigma_i^\rho, \sigma_i^{\rho u}, \sigma_i^{E})$.
For the multi-dimensional notation, we refer to \cite{ThomannZenkKlingenberg2019}.
The slopes $\sigma_i$ are computed from the neighbouring cells using a limiter function. 
To have a seconder order extension that preserves the positivity properties of the first order scheme, see Theorem \ref{theo:Pos_first}, we choose the minmod limiter defined as
\begin{equation*}
	\text{minmod}(x,y) = 
	\begin{cases}
	\min(x,y) &\text{if } x,y \geq 0 \\
	\max(x,y) &\text{if } x,y \leq 0 \\
	0 &\text{else} 
	\end{cases}
\end{equation*} 
and define the slopes as
\begin{equation}
	\sigma_i = \text{minmod}(\frac{w_i^n - w_{i-1}^n}{\Delta x}, \frac{w_{i+1}^n - w_i^n}{\Delta x}).
\end{equation}
The same procedure is applied on $\psi_i^{(1)}$. 
Since we are using the minmod limiter on the conservative variables to determine the slope, we immediately get for $w_i^n \in \Omega_{phy} \cap \Omega_{wp}$ that for the interface values holds
\begin{equation*}
	w_{i-1/2}^+, w_{i+1/2}^- \in \Omega_{phy} \cap \Omega_{wp}
\end{equation*}
and the expansion of the interface values $\psi_{i-1/2}^{(1),+},\psi_{i+1/2}^{(1),-}$ with respect to the Mach number is preserved.
This means by Lemma \ref{lem:RS_pos1} that the Riemann problem applied on the interface values $w_{i+1/2}^{\mp},\psi_{i+1/2}^\mp$ still ensures the positivity of density and internal energy.
By Theorem \ref{theo:Pos_first} the first order scheme has the positivity property and the second order scheme \eqref{sys:Time_int2_phys} is a convex combination of states in $\Omega_{phy}$.
From this we have $w_i^{n+1} \in \Omega_{phy}$.
We have thus proven the following result for the second order scheme:
\begin{theorem}[Positivity property 2]
\label{theo:Pos_second}
Let the initial state be given as $$w_i^n \in \Omega_{phy} \cap \Omega_{wp}$$ with the boundary condition described in \eqref{bc:pressure}. 
Then under the Mach number independent CFL condition $\frac{\Delta t}{\Delta x} \underset{i}{\max}|\lambda^\pm(w_i^n)| < \frac{1}{2 \cdot 2d}$, the numerical scheme defined by \eqref{sys:Time_int2}/\eqref{sys:Time_int2_phys} preserves the positivity of density and internal energy, that is 
$$
	w_{i}^{n+1} \in \Omega_{wp}.
$$
\end{theorem}

\section{Numerical results}
\label{sec:NumRes}
In the following section, we numerically validate the theoretical properties of the proposed scheme. 
Throughout the test cases, we assume an ideal gas law with the equation of state given by
\begin{align*}
	p = (\gamma - 1) \rho e.
\end{align*}
For solving the implicit non-symmetric linear system given by \eqref{eq:update_impl}, we use the GMRES algorithm combined with a Jacobi diagonal scaling preconditioner for periodic boundary conditions used in the test cases in Section 7.2 and an preconditioner based on an incomplete LU decomposition for von Neumann boundary conditions used in the test cases from Section 7.1.
\subsection{Shock test cases}
To verify that our scheme captures discontinuities accurately, we perform a SOD shock tube test \cite{Sod1978} in the regime $M=1$ and a Mach number dependent shock test case taken from \cite{AbbateIolloPuppo2017}. 
\subsubsection{SOD shock tube test}
The computational domain for the SOD shock tube test \cite{Sod1978} is $[0,1]$ and the initial data is given using  $\gamma = 1.4$ by 
\begin{align*}
\arraycolsep=1.4pt\def\arraystretch{1.5}
\begin{array}{clccl}
	\rho_L &= 1 ~\frac{kg}{m^3},&\hspace{1cm}&\rho_R &= 0.125 ~\frac{kg}{m^3}, \\
	u_L &= 0 ~\frac{m}{s}, && u_R &= 0 ~\frac{m}{s}, \\
	p_L &= 1 ~\frac{kg}{m s^2}, &&p_R &= 0.1 ~\frac{kg}{m s^2},
	\end{array}
\end{align*}
where we place the initial discontinuity at $x = 0.5$. 
Since the regime is compressible, we set $M=1$ and the initial data is given in dimensional form.
This test also demonstrates that the collapse of the eigenvalues $\lambda^\pm$ and $\lambda_M^\pm$ in the case of $M=1$ as mentioned in Remark \ref{rem:RS_M=1} is not problematic.
We see in Figure \ref{fig:SOD} that the first order as well as the second order scheme captures the shock positions correctly.
As expected, the second order scheme is more accurate than the first order scheme.

\begin{figure}[htpb]
	\centering
	\includegraphics[scale=0.34]{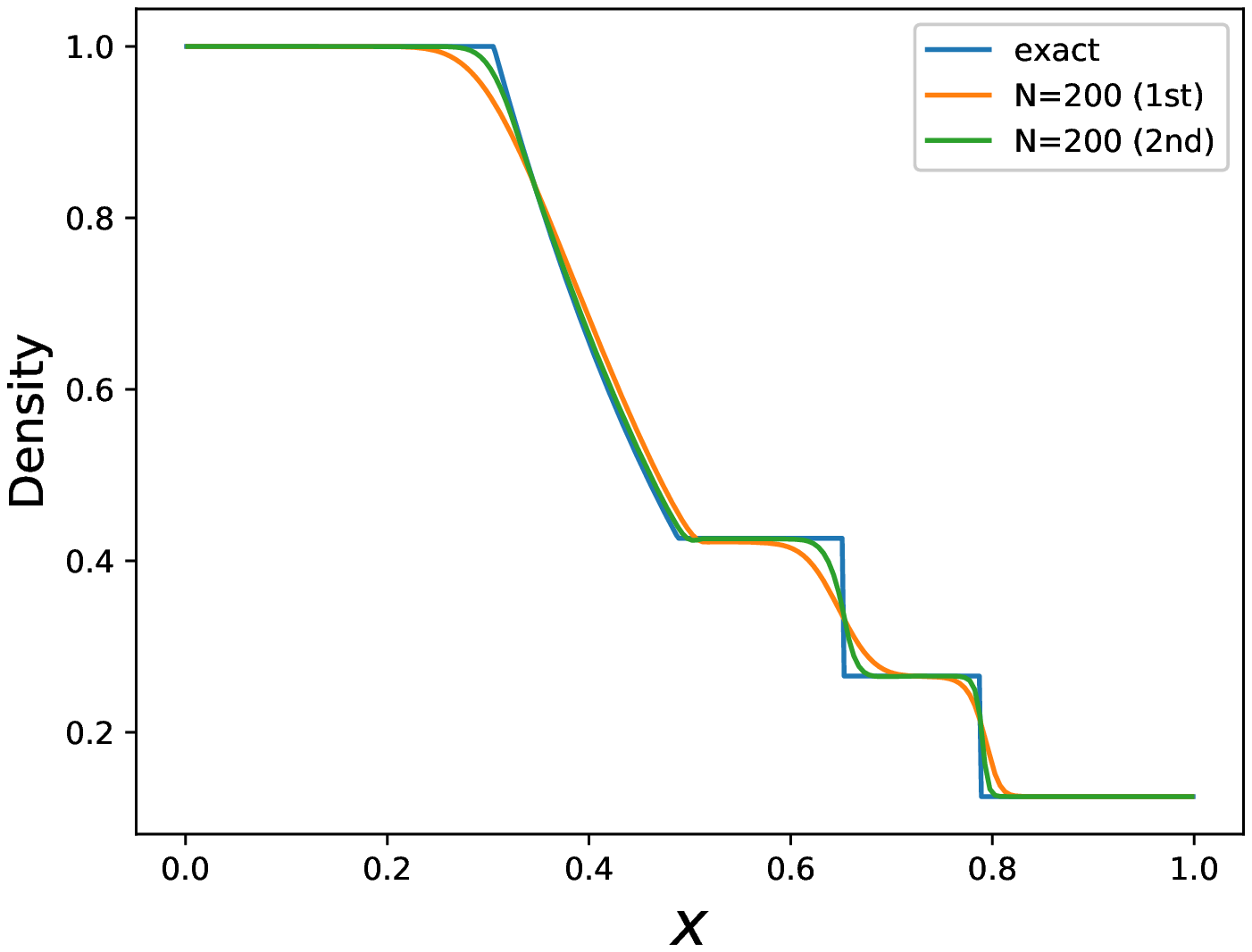}\hskip-4mm
	\includegraphics[scale=0.34]{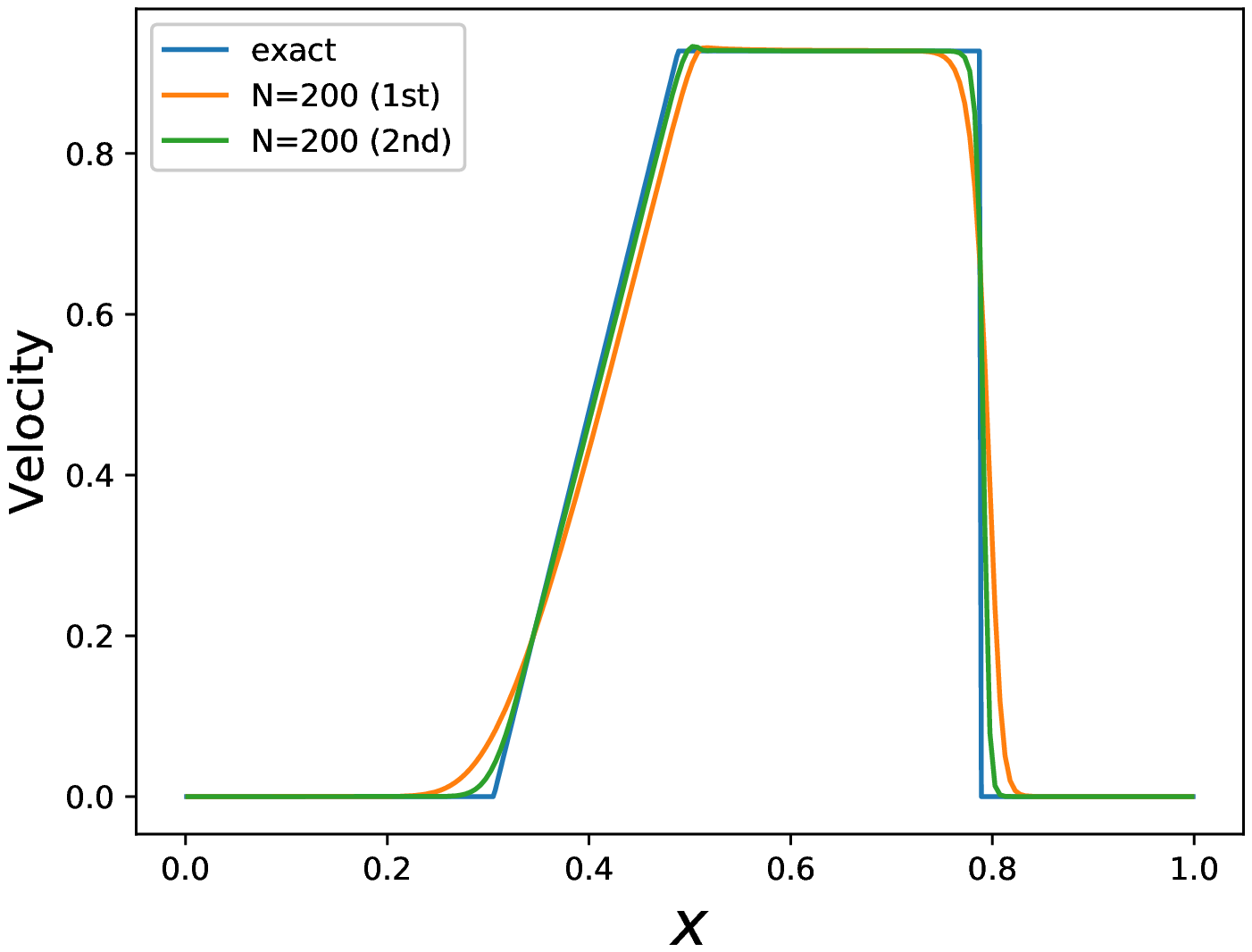}\hskip-4mm
	\includegraphics[scale=0.34]{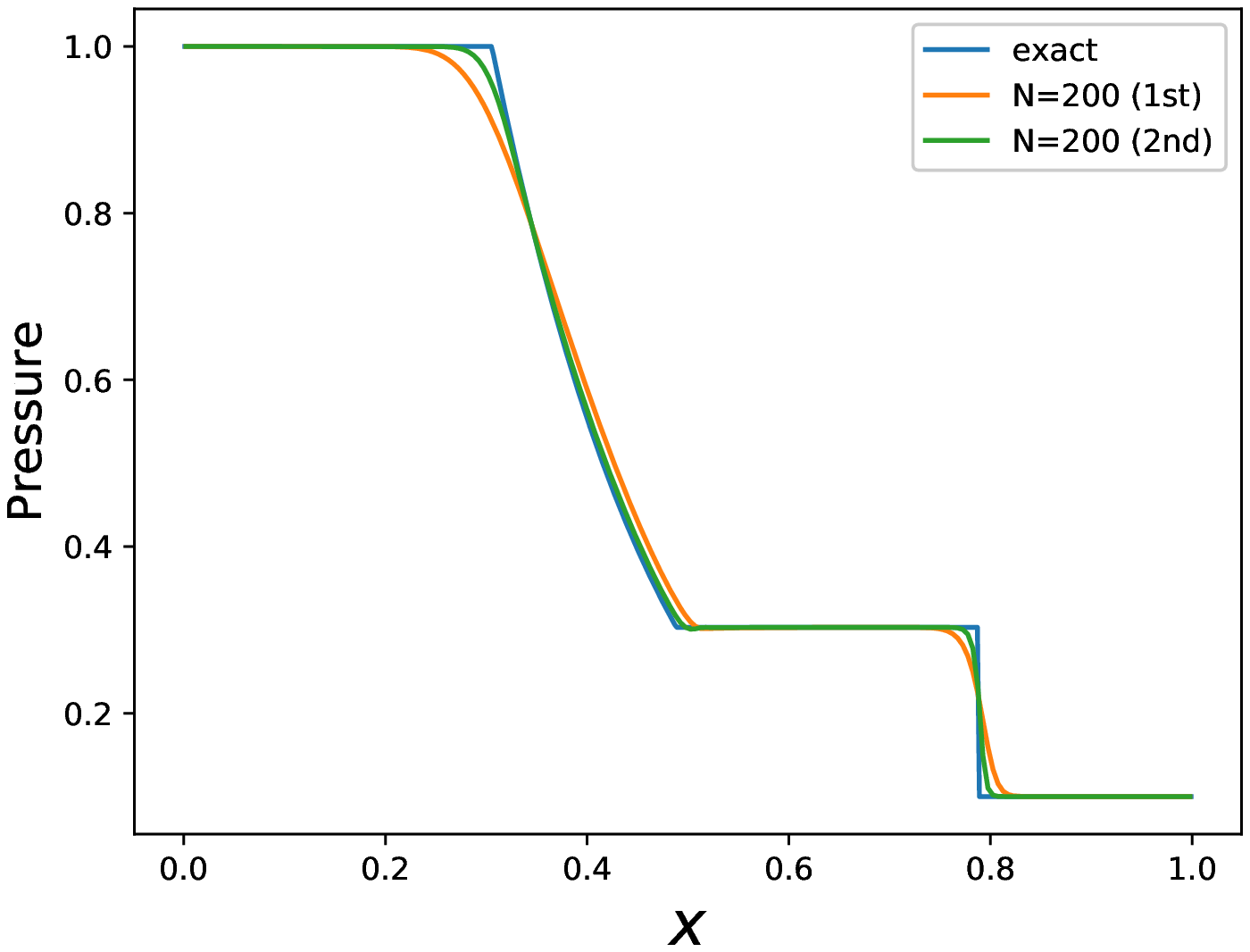}
	\caption{SOD: First order scheme, $T = 0.1644, \gamma = 1.4$}
	\label{fig:SOD}
\end{figure}

\subsubsection{Mach number dependent shock problem}
The following test case is a low Mach flow, where the contact wave travels with the Mach number $M=6.2\cdot 10^{-3}$. 
The initial data is taken from \cite{AbbateIolloPuppo2017} and is given as
\begin{align}
\label{init:Shock2}
\arraycolsep=1.4pt\def\arraystretch{1.5}
\begin{array}{clccl}
\rho_L &= 1 ~\frac{kg}{m^3},&\hspace{1cm}&\rho_R &= 1 ~\frac{kg}{m^3}, \\
u_L &= 0 ~\frac{m}{s}, && u_R &= 0.008 ~\frac{m}{s}, \\
p_L &= 0.4 ~\frac{kg}{m s^2}, &&p_R &= 0.399 ~\frac{kg}{m s^2}.
\end{array}
\end{align}
The discontinuity is placed at $x_0 = 0.5$ on the domain $[0,1]$ with $\gamma = 1.4$ and the final time is given by $T = 0.25s$.
To transform the initial data \eqref{init:Shock2} into non-dimensional quantities, we define for a variable $\phi$ the relation $\phi = \phi_r ~\hat \phi$, where $\phi$ denotes the dimensional variable, $\phi_r$ the reference value which contains the units and $\hat \phi$ the non-dimensional quantity.
For the reference values we have the following relations
\begin{equation}
	t_r = \frac{x_r}{t_r}, ~ p_r = \rho_r c_r^2, ~ c_r = \frac{u_r}{M}.
\end{equation} 
We choose the scaling in space to be $x_r = 1 m$, the scaling in density to be $\rho_r = 1 \frac{kg}{m^3}$ and the scaling in velocity to be $u_r = M \frac{m}{s}$.
This yields the following scaling of the sound speed $c_r = 1 \frac{m}{s}$, the time $t_r = M s$ and the pressure $p_r = 1 \frac{kg}{m s^2}$. 
Then the non-dimensional initial data is given as 
\begin{align}
\label{init:Shock2_nondim}
\arraycolsep=1.4pt\def\arraystretch{1.5}
\begin{array}{clccl}
\rho_L &= 1,&\hspace{1cm}&\rho_R &= 1, \\
u_L &= 0, && u_R &= 0.008/M, \\
p_L &= 0.4, &&p_R &= 0.399.
\end{array}
\end{align}
In the simulation, we choose $M = 6.2\cdot 10^{-3}$ which is the Mach number on the contact wave.

In Figure \ref{fig:SOD2} we show the influence of the space and time step on the density profile computed with the first order scheme (IMEX1) and the second order scheme (IMEX2). 
The contact wave is always reproduced sharply whereas the acoustic waves are smoothed since the time step is chosen according to the CFL restriction \eqref{eq:cfl} associated with the contact wave.

In Figure \ref{fig:SOD2_Compare} our results computed with IMEX1 and IMEX2 are plotted against the implicit Jin Xin relaxation scheme presented in \cite{AbbateIolloPuppo2017} and an explicit upwind scheme. 
Our results obtained with the IMEX schemes are in good agreement with the results of the implicit scheme. 
Since the time step for the explicit upwind scheme depends on the Mach number dependent acoustic waves, all waves are resolved. 
This is rather costly since it results in a very small time step.
The implicit scheme is unconditionally stable and the time step can be chosen with respect to the desired accuracy of the numerical solution.
\begin{figure}[htpb]
	\centering
	\includegraphics[scale=0.6]{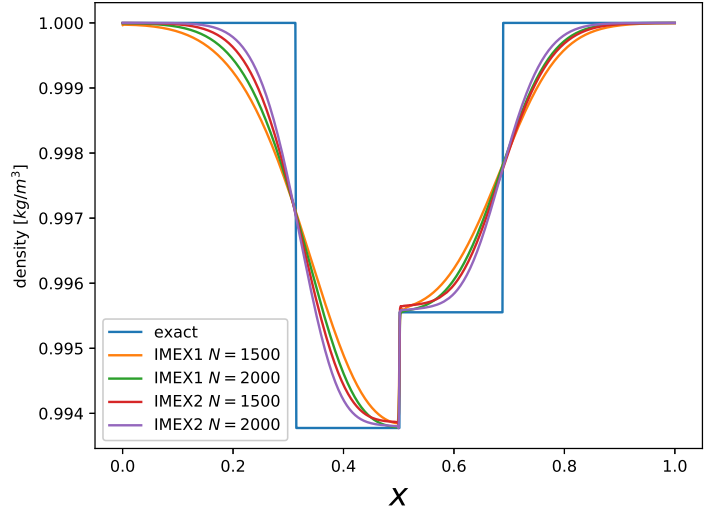}\hskip-1.5mm
	\caption{Density profile for the Mach number dependent test case with different number of grid point and time steps. Time steps for IMEX1 $\Delta t = 7.8 \cdot 10^{-3}s$ (blue), $\Delta t = 5.8 \cdot 10^{-3}s$ (orange), time steps for IMEX2 $\Delta t = 3.9 \cdot 10^{-3}s$ (green), $\Delta t = 2.9 \cdot 10^{-3}s$ (red)}
	\label{fig:SOD2}
\end{figure}
\begin{figure}[htpb]
	\centering
	\includegraphics[scale=0.33]{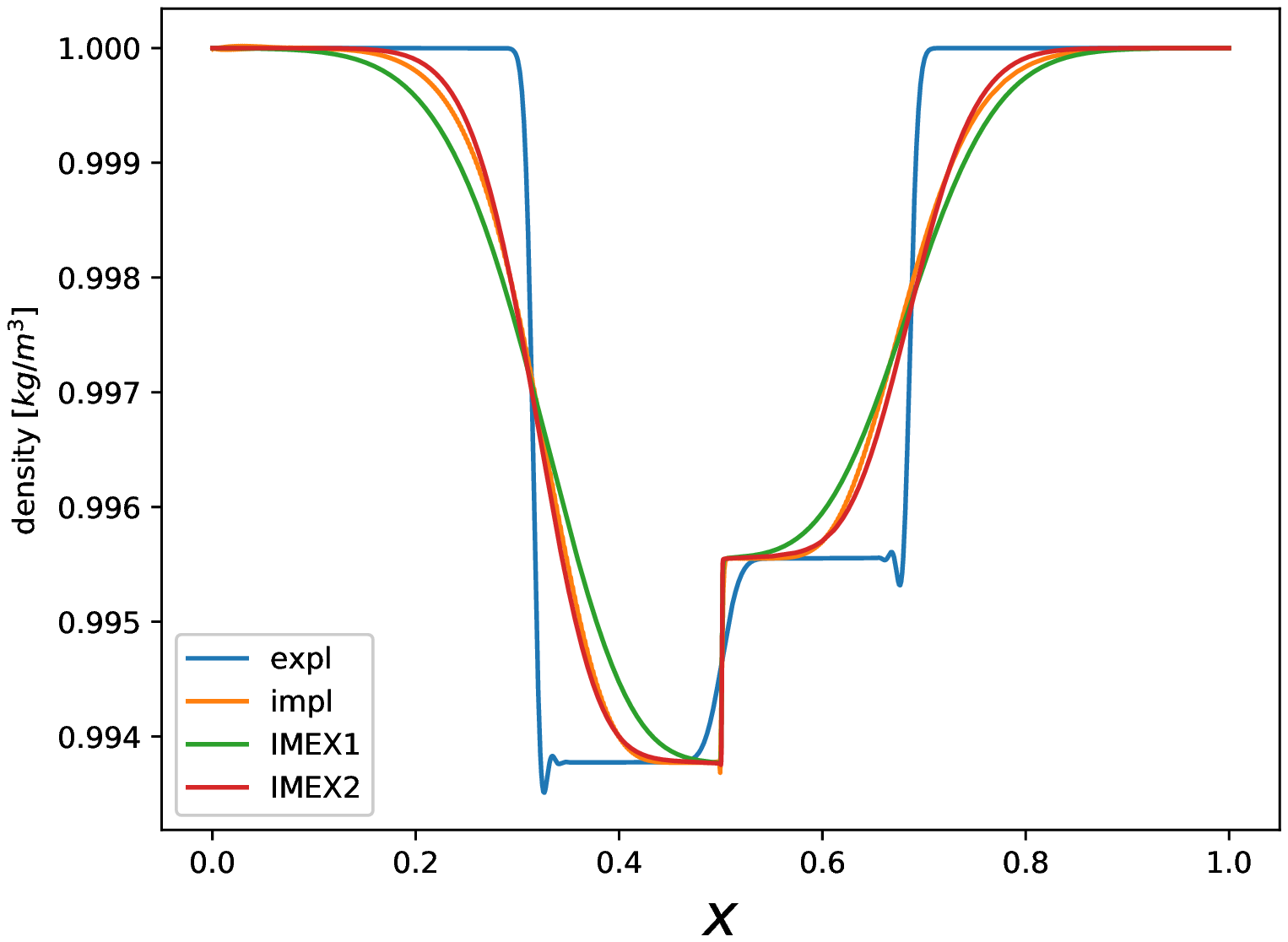}\hskip-1.5mm
	\includegraphics[scale=0.33]{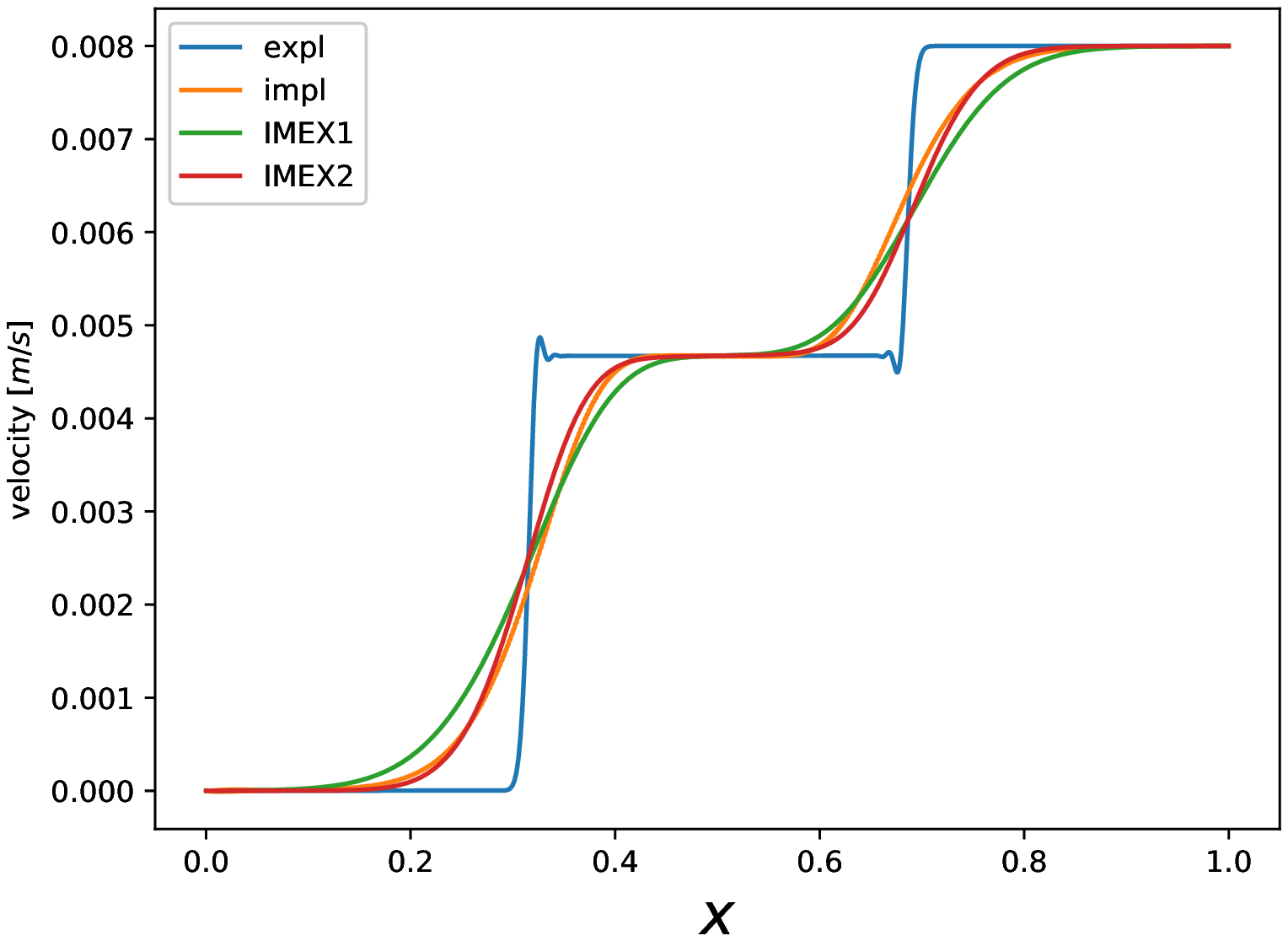}\hskip-1.5mm
	\includegraphics[scale=0.33]{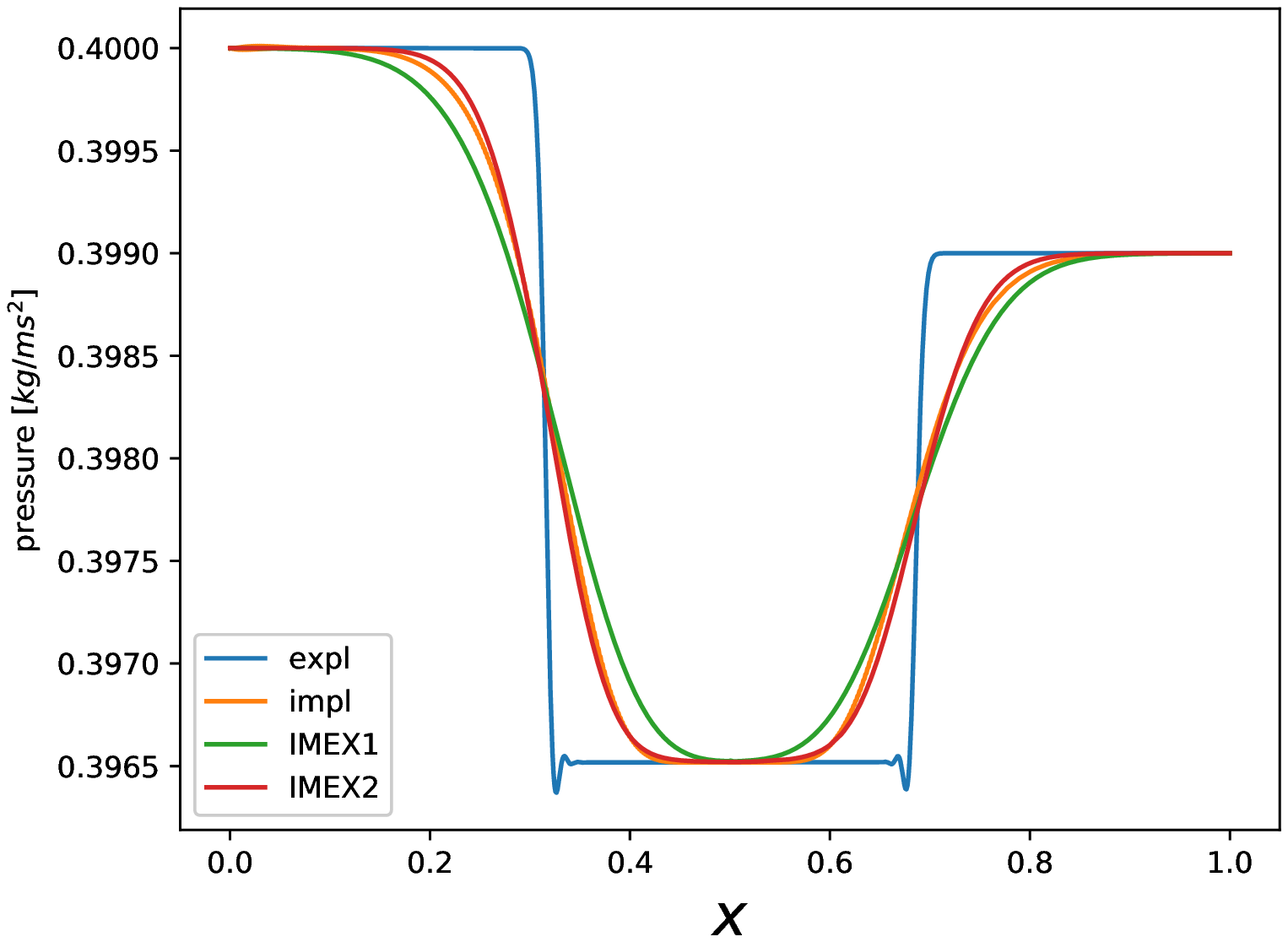}
	\caption{Mach dependent shock test case: Time steps are given by $\Delta t = 5.3 \cdot 10^{-4}s$ (expl.), $\Delta t = 2.2 \cdot 10^{-2}s$ (impl.), $\Delta t = 5.0 \cdot 10^{-3}s$ (IMEX1), $\Delta t = 2.5 \cdot 10^{-3}s$ (IMEX2).}
	\label{fig:SOD2_Compare}
\end{figure}

\subsection{Gresho Vortex test}
In order to demonstrate the low Mach properties, we calculate the solution to the Gresho Vortex test as given in \cite{MiczekRoepkeEdelmann2015}. 
The Gresho vortex is a steady state solution of the compressible Euler equations and the velocity field is divergence free.

The initial angular velocity in $\frac{m}{s}$ is given by 
\begin{align*}
u_\phi &= 
\begin{cases}
5r &\text{ for } 0 \leq r < 0.2\\
2 - 5r &\text{ for } 0.2 \leq r < 0.4\\
0 &\text{ for } 0.4 \leq r
\end{cases},
\end{align*}
where $r = \sqrt{(x-x_0)^2 + (y - y_0)^2}$ with $x_0 = 0.5, 
y_0 = 0.5$ on a computational domain of $[0,1]$ and $\gamma = 5/3$.
The pressure distribution in $\frac{kg}{ms^2}$ is given by 
\begin{align*} 
p &=
\begin{cases}
p_0 + 12.5 r^2 &\text{ for } 0 \leq r < 0.2\\
p_0 + 12.5 r^2 + 4(1 - 5r - \log(0.2) + \log(r)) &\text{ for }0.2 \leq r < 0.4 \\
p_0 - 2 + 4 \log(2) &\text{ for } 0.4 \leq r
\end{cases}\\
\end{align*}
with $p_0 = \frac{\rho_0 u_{\Phi,\max}^2}{\gamma M^2} \frac{kg}{m s^2}$, where $\rho_0 = 1$. 
The initial density is given by $\rho = 1\frac{kg}{m^3}$ and we transform the initial condition in non-dimensional quantities by using $x_r = 1 m$, $\rho_{r} = 1 \frac{kg}{m^3}$,  $u_{r} = 2 \cdot 0.2~\pi\frac{m}{s}$, $p_{r} = \frac{\rho_0 u_0^2}{\gamma M^2} \frac{kg}{m s^2}$ and $t_r=1\frac{m}{u_r}$. 
The computational domain is given by $[0,1]\times[0,1]$ and we use a $40\times40$ grid with periodic boundary conditions.
The results for a full turn of the vortex using the IMEX2 scheme together with the initial distribution of the Mach number are given in Figure \ref{fig:Gresho}.
We see that even for a low resolution, the solution at $T=1$ shows only little dissipation throughout all tested Mach numbers. 
To further check the quality of the numerical simulation, we monitor the loss of kinetic energy. 
Since the Gresho vortex is a stationary solution, the kinetic energy should be preserved. 
In Figure \ref{fig:Ekin} we show the ratio between the initial kinetic energy $E_{kin,0}$ and the kinetic energy after each time step $E_{kin,t}$ for the Mach numbers $M=10^{-2}, 10^{-3}$.
The graphs for the different Mach numbers are indistinguishable which shows that the loss of kinetic energy does not depend on the chosen Mach number but depends on the chosen space discretization and time-step.

\begin{figure}[htpb]
	\centering
	\includegraphics[scale=0.4]{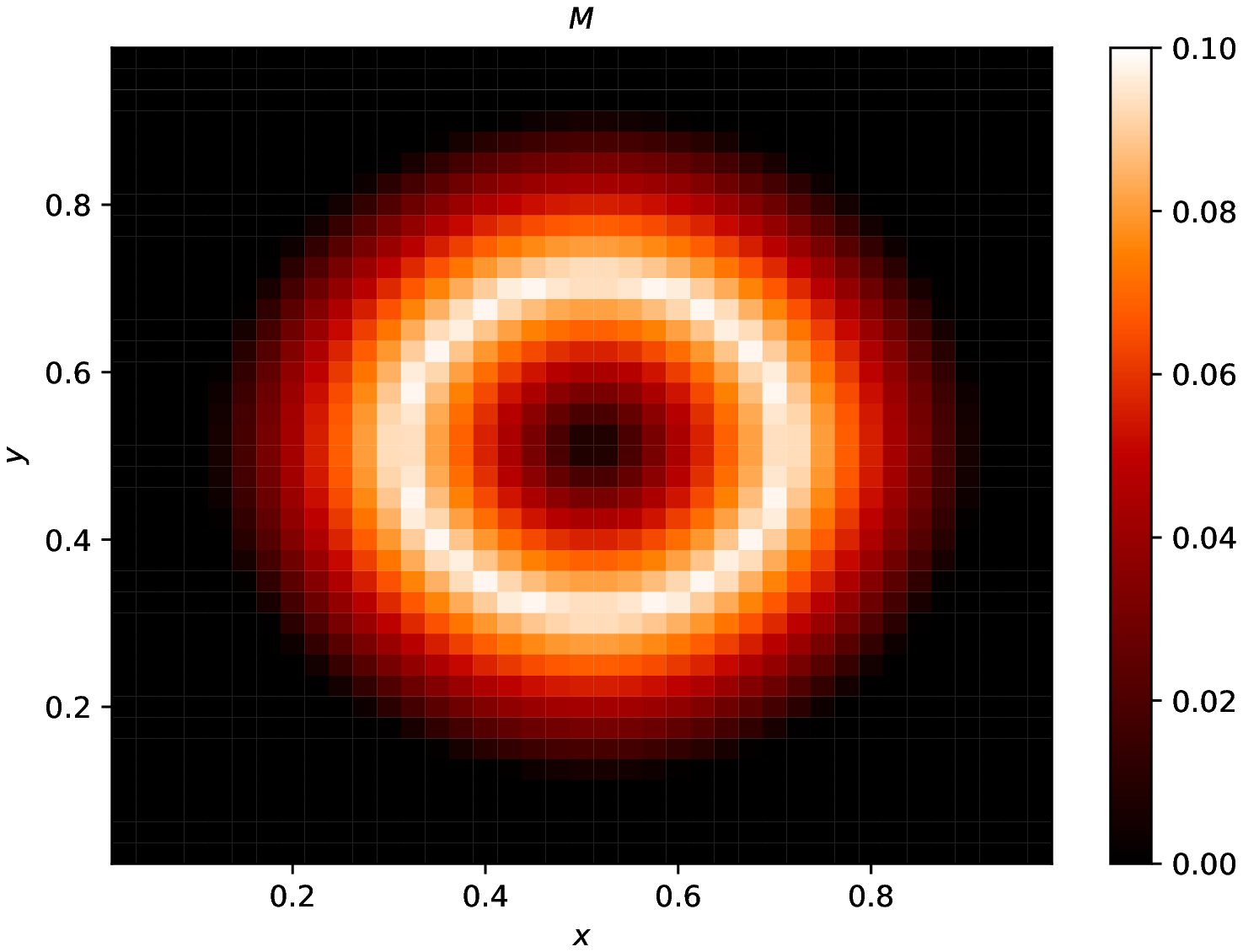}
	\includegraphics[scale=0.4]{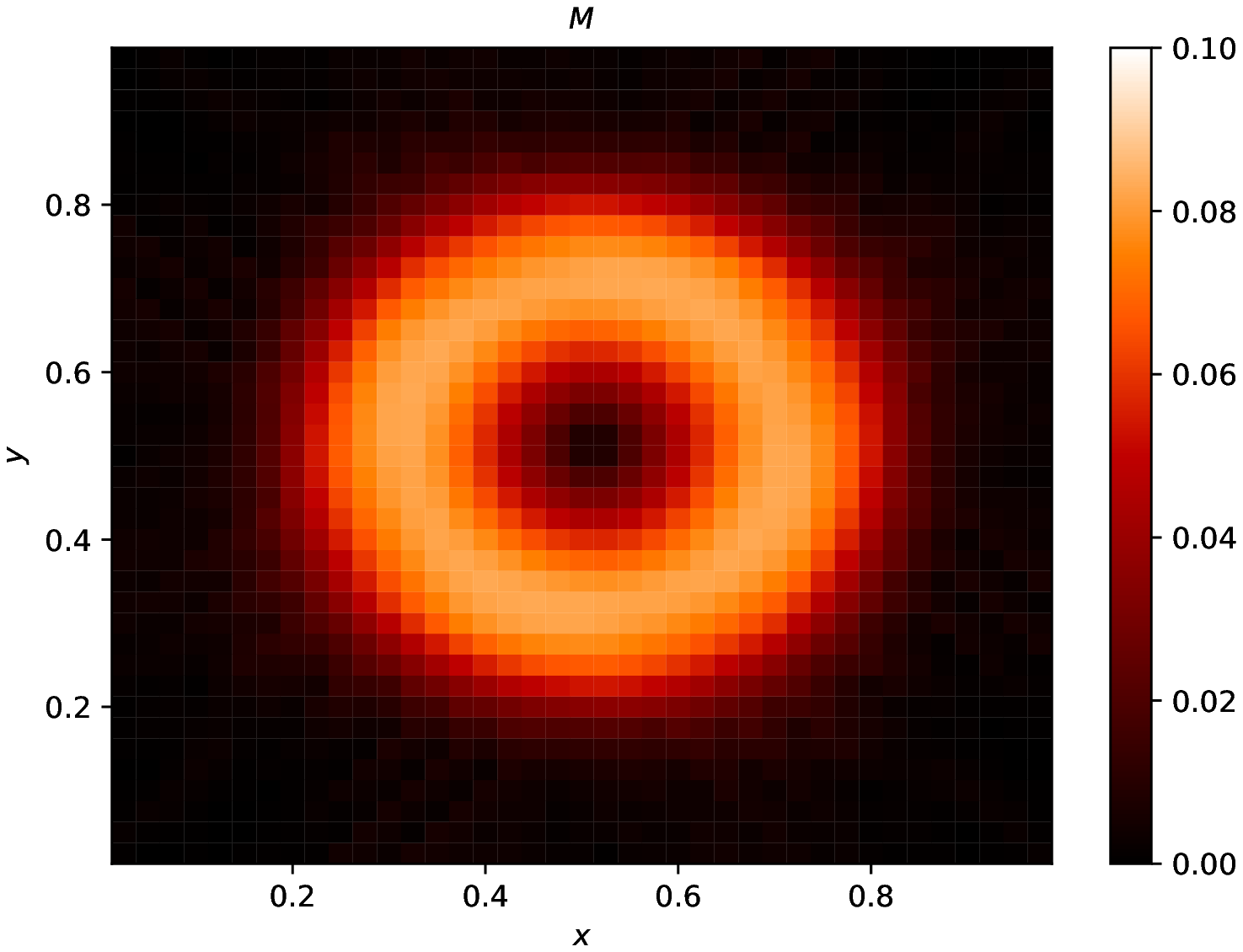}\\
	\includegraphics[scale=0.4]{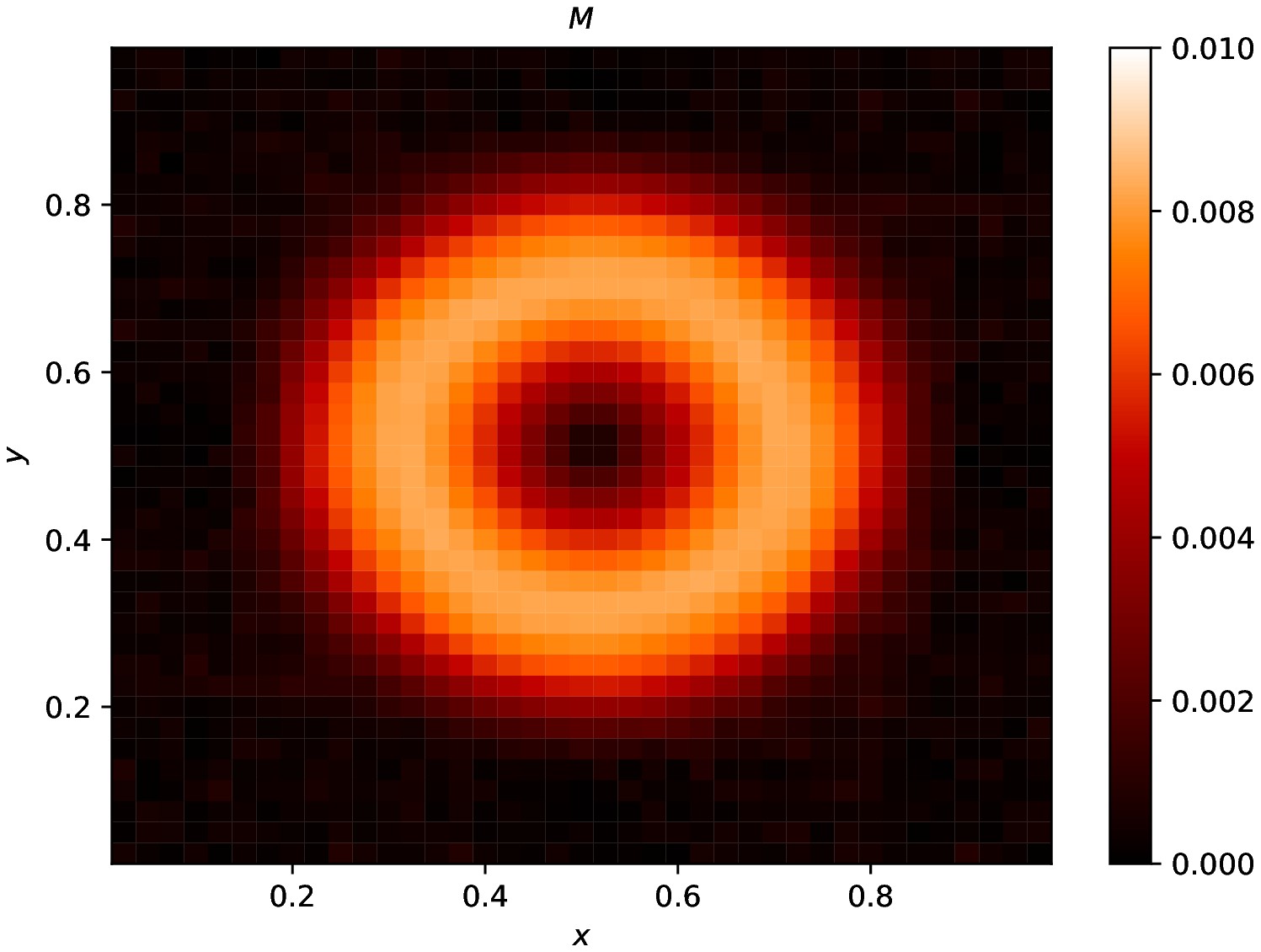}
	\includegraphics[scale=0.4]{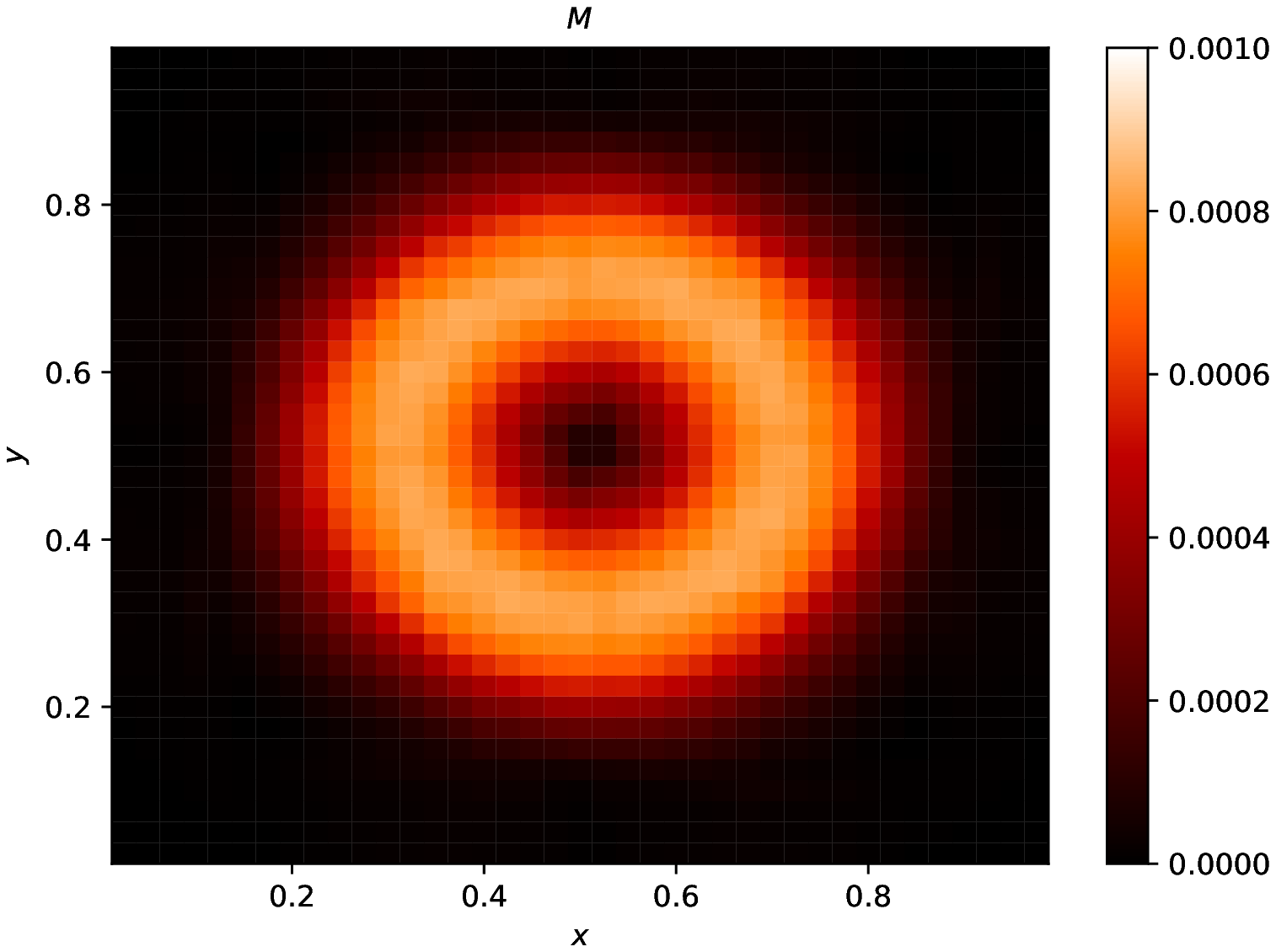}
	\caption{Mach number distribution for different maximal Mach numbers. Top left: Initial state for $M=10^{-1}$. Top right: $M=10^{-1}$, bottom left: $M=10^{-2}$, bottom right: $M=10^{-3}$ at $t=1$.}
	\label{fig:Gresho}
\end{figure}

\begin{figure}[htpb]
	\centering
	\includegraphics[scale=0.5]{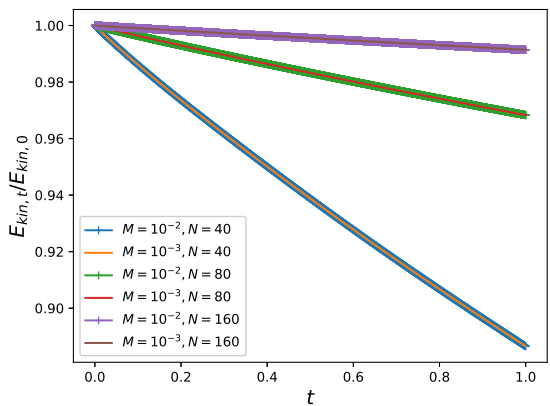}
	\caption{Loss of kinetic energy for different grids and Mach numbers after one full turn of the vortex (non-dimensional).}
	\label{fig:Ekin}
\end{figure}

\subsection{A smooth Gresho Vortex test}
In this section we want to derive a test case to numerically validate the second order accuracy of the second order extension of Section \ref{sec:Second}. 
Unfortunately the Gresho vortex is not a smooth enough solution to test second order accuracy, since its velocity profile is only continuous but not continuously differentiable. Therefore we propose a smoothed velocity profile with which we calculate a pressure profile to gain a stationary vortex. 
A twice continuously differentiable angular velocity in $\frac{m}{s}$ with $u_{\Phi,\max} =1$ and $u_\Phi(0) = 0$ and $u_\Phi(0.4) = 0$ is given by
	\begin{align}
u_{\Phi} = 
\begin{cases}
75 r^2 - 250 r^3, &\text{ for } 0 \leq r < 0.2 \\
-4 + 60 r - 225 r^2 + 250 r^3, &\text{ for } 0.2 \leq r < 0.4 \\
0 &\text{ for } 0.4 \geq r
\end{cases}
\end{align}
with the radius $r = \sqrt{(x - 0.5)^2 + (y - 0.5)^2}$.
The profile can be easily modified to be a $C^k$ function, where $k\in \mathbb{N}$ denotes the degree of continuous differentiability. 
Under the condition that the centrifugal forces are balanced that is $\partial_r p = \frac{u_\Phi(r)^2}{r}$, we can calculate the pressure in $\frac{kg}{m s^2}$ as
\begin{align}
p = 
\begin{cases}
p_0 + 1406.25 r^4 - 7500 r^5 + (10416 + \frac{2}{3}) r^6 &\text{ for } 0 \leq r < 0.2 \\
p_0 + p_2(r) &\text{ for }  0.2 \leq r < 0.4 \\
p_0 + p_2(0.4) &\text{ for } 0.4 \geq r
\end{cases}
\end{align}
where 
\begin{align}
p_2(r) =& 65.8843399322788 - 480r  + 2700 r^2 - (9666 + \frac{2}{3}) r^3 \\
&+ 20156.25 r^4 - 22500 r^5 + (10416 + \frac{2}{3}) r^6 + 16 \ln(r).
\end{align}
As in the Gresho test case, the background pressure is scaled with the Mach number as
\begin{equation}
p_0 = \frac{\rho_0 u_{\Phi,\max}^2}{\gamma M^2} \frac{kg}{m s^2}.
\end{equation}
To transform the dimensional data into non-dimensional quantities, we use the same reference values as in the Gresho vortex test case.

To show the accuracy of the IMEX2 scheme, we compute the solution of the smooth Gresho vortex test on the domain $[0,1]^2$ with periodic boundary conditions. 
In Table \ref{tab:ConvSmoothGresho} the $L^1$-error between the numerical solution at $T=0.05$ and the initial configuration in non-dimensional quantities as well as the convergence rates are displayed for $M = 10^{-1}, 10^{-2},10^{-3}$. 
It can be seen that we reach the expected accuracy independently of the chosen Mach number.
\begin{table}[htpb]
	\begin{center}
		\renewcommand{\arraystretch}{1.2}
		\begin{tabular}{cc  cc cc cc cc} 
			$M$&$N$ & $\rho$ & & $u_1$ & & $u_2$ & & $p$ & \\
			\hline \hline 
			\multirow{4}{*}{$10^{-1}$}& 20& 1.810$ \cdot 10^{-3}$  &---& 1.729$ \cdot 10^{-2}$  &---&  1.729$ \cdot 10^{-2}$  &---& 1.921$ \cdot 10^{-3}$ &---\\
			&40& 3.705$ \cdot 10^{-4}$ &2.288&  5.070$ \cdot 10^{-3}$ &1.770&  5.070$ \cdot 10^{-3}$ &1.770&  3.956$ \cdot 10^{-4}$ &2.279\\ 
			&60& 1.246$ \cdot 10^{-4}$ &2.688&  2.403$ \cdot 10^{-3}$  &1.842& 2.403$ \cdot 10^{-3}$ &1.842&  1.343$ \cdot 10^{-3}$  &2.665\\
			&80& 5.510$ \cdot 10^{-5}$ &2.835&  1.396$ \cdot 10^{-3}$  &1.887& 1.396$ \cdot 10^{-3}$ &1.887&  5.922$ \cdot 10^{-5}$&2.845\\
			\hline
			\multirow{4}{*}{$10^{-2}$}&20 &1.812$ \cdot 10^{-3}$  &---& 1.731$ \cdot 10^{-2}$ &---&  1.731$ \cdot 10^{-2}$  &---& 1.912$ \cdot 10^{-3}$ &---\\
			&40& 3.582$ \cdot 10^{-4}$ &2.339&  5.057$ \cdot 10^{-3}$  &1.775& 5.057$ \cdot 10^{-3}$ &1.775&  3.781$ \cdot 10^{-4}$  &2.337\\
			&60& 1.162$ \cdot 10^{-4}$ &2.777&  2.402$ \cdot 10^{-3}$  &1.837& 2.402$ \cdot 10^{-3}$ &1.837&  1.226$ \cdot 10^{-4}$  &2.777\\ 
			&80& 4.881$ \cdot 10^{-5}$ &3.014&  1.386$ \cdot 10^{-3}$  &1.910& 1.386$ \cdot 10^{-3}$ &1.910&  5.151$ \cdot 10^{-5}$&3.014\\
			\hline
			\multirow{4}{*}{$10^{-3}$}& 20& 1.811$ \cdot 10^{-3}$ &---& 1.731$ \cdot 10^{-2}$ &---&   1.731$ \cdot 10^{-2}$  &---& 1.912$ \cdot 10^{-3}$  &---\\
			&40& 3.580$ \cdot 10^{-4}$ &2.339&  5.057$ \cdot 10^{-3}$  &1.775& 5.057$ \cdot 10^{-3}$&1.775&   3.778$ \cdot 10^{-4}$  &2.339\\
			&60& 1.162$ \cdot 10^{-4}$  &2.775& 2.402$ \cdot 10^{-3}$  &1.836& 2.402$ \cdot 10^{-3}$ &1.836&  1.227$ \cdot 10^{-4}$  &2.775\\ 
			&80& 4.875$ \cdot 10^{-5}$ &3.019&  1.382$ \cdot 10^{-3}$ &1.920&   1.382$ \cdot 10^{-3}$ &1.920&  5.146$ \cdot 10^{-5}$ &3.019\\
			\hline
		\end{tabular}
	\end{center}
	\caption{$L^1$-error and convergence rates for the solution of the smooth Gresho vortex test at $T = 0.05$ (non-dimensional).}
	\label{tab:ConvSmoothGresho}
\end{table}

\section{Conclusion and future developments} 
We have proposed an all-speed IMEX scheme for the Euler equations of gas dynamics which is based on Suliciu relaxation model.
The proposed IMEX scheme is an improvement of explicit schemes since the time step is restricted by the material and not the acoustic wave speeds. 
It is also an improvement of implicit schemes since the implicit part consists only of one linear equation and can be solved very efficiently. 
The scheme has the correct numerical viscosity for all Mach numbers as shown by the Gresho vortex test case, can capture the correct shock positions as shown by the SOD shock test case as well as for Mach number dependent shock problems. 
In addition it is positivity preserving and shows the expected second order convergence rates.

In a future work, the scheme will be extended to the Euler equations with a gravitational source term. 
The aim is to design a multi-dimensional well-balanced scheme, that inherits all the properties of the IMEX scheme presented here.

\section*{Aknowledgements}
	G. Puppo ackwowledges the support by the GNCS-INDAM 2019 research project and A. Thomann the support of the INDAM-DP-COFUND-2015, grant number 713485.
Some material of this work has been improved during the SHARK-FV conference (Sharing Higher-order Advanced Know-how on Finite
Volume http://www.SHARK-FV.eu/) held in 2019.
\bibliographystyle{unsrt}
\bibliography{lit.bib}
\end{document}